\documentclass[final,authoryear]{elsarticle}

\usepackage{amssymb}

\usepackage{lineno}

\biboptions{sort&compress}

\usepackage{amsthm,amsmath,amssymb,amsfonts,delarray}
\usepackage{fullpage}
\usepackage{color}
\usepackage{enumerate}
\usepackage{comment}
\usepackage{bbm,multirow}
\usepackage{appendix}
\usepackage{subcaption}
\usepackage{booktabs}
\usepackage[dvipsnames]{xcolor}
\usepackage{tikz}
\usepackage{natbib}
\newcolumntype{M}[1]{>{\centering\arraybackslash}m{#1}}

\newtheorem{theorem}{Theorem}[section]

\newtheorem{lemma}{Lemma}[section]

\newtheorem{condition}{Condition}[section]
\theoremstyle{definition}

\newcommand{\R}{\mathbb{R}}

\newcommand{\Ep}{\mathbb{E}}
\renewcommand{\Pr}{\mathbb{P}}
\renewcommand{\tilde}{\widetilde}
\renewcommand{\hat}{\widehat}

\newcommand{\diff}{\mathrm{d}}

\DeclareMathOperator{\Var}{Var}

\allowdisplaybreaks[4]

\usepackage{mdframed}
\usepackage{cancel}
\newcommand{\underk}{\underline{k}}
\newcommand{\norm}[1]{\|#1\|}

\journal{Statistics and Probability Letters}

\begin{document}

\begin{frontmatter}

\title{Dependence of variance on covariate design in nonparametric link regression}

\author[ism,riken]{\corref{okuno_ca}\fnref{okuno_email}Akifumi Okuno}
\author[ism]{\fnref{yano_email}Keisuke Yano}
\cortext[okuno_ca]{Corresponding author}
\fntext[okuno_email]{okuno@ism.ac.jp}
\fntext[yano_email]{yano@ism.ac.jp}
\address[ism]{The Institute of Statistical Mathematics. 10-3 Midori cho, Tachikawa City, Tokyo, 190-8562, Japan.}
\address[riken]{RIKEN Center for Advanced Intelligence Project. Nihonbashi 1-4-1 Nihonbashi, Chuo-ku, Tokyo, 103-0027, Japan.}

\begin{abstract}
This paper discusses a design-dependent nature of variance in nonparametric link regression aiming at predicting a mean outcome at a link, i.e., a pair of nodes, based on currently observed data comprising covariates at nodes and outcomes at links. 
\end{abstract}


\end{frontmatter}


\section{Introduction}
Binary link prediction in node-attributed graphs~\citep{taskar2003link,menon2011link} is the task of predicting the link existence from the corresponding node covariates. It has attracted widespread and long-standing interest in a wide range of applied areas such as social science studies~\citep{gong2014joint}. 
The binary link prediction is generalized to similarity learning~\citep{pmlr-v38-liu15,kulis2012metric,Pinheiro_2018_CVPR}, which aims at predicting possibly non-binary outcomes (representing similarity) defined between two node covariates. 
Similarity learning is classified into two types depending on whether node covariates come from a single domain~\citep{pmlr-v38-liu15} or two different domains~\citep{Pinheiro_2018_CVPR}. The former setting is theoretically interesting as it assumes symmetry of similarity. 
In the single-domain setting, \citet{okuno2020hyperlink} and \citet{graham2020dyadic} consider parametric regression for predicting real-valued outcomes from covariates. Regression for learning similarity in the single-domain setting is called link regression.

This paper focuses on a nonparametric approach to link regression that predicts mean outcomes at links of size $n(n-1)/2$ based on node covariates of size $n$. 
We show that the decay rate of the deterministic bias with respect to a bandwidth parameter and the sample size remains the same regardless of whether the designs of node covariates are fixed (i.e., fixed design) or randomly obtained (i.e., random design), whereas that of the stochastic variance drastically differs depending on covariates' designs. 
This dependence comes from the conditional bias due to the randomness of covariates in random design cases. The conditional bias behaves like the non-degenerate $U$-statistics~\citep{lee1990u} and is controlled by the sample size of independent node covariates. 
Figure~\ref{fig:comparison}(\subref{subfig:NLR}) showcases a numerical result demonstrating this theoretical finding in comparison to conventional nonparametric regression. At the same time, this study demonstrates that there exists a threshold value, below which the difference between two designs disappears. In particular, the variance decay rates in both designs match for the bandwidth alleviating the bias-variance trade-off.

The distinction between the fixed and random designs in regression has been addressed in the literature on statistics and machine learning. Fixed design analysis is common when covariates are physically controlled, while random design analysis is conducted when the values of covariates are unpredictable. These designs are fundamentally different as discussed in \citet{gyorfi2002distribution}. 
\citet{brown1990ancillarity} has displayed the design-dependent nature of the statistical decision theory in parametric regression models. 
\citet{Buja2019a} has demonstrated that 
the randomness of covariates produces additional variance of parametric regression estimates.
However, in nonparametric regression, the difference in the decay rates of bias and variance between fixed and random designs is known to appear only in a multiplicative constant (e.g., \citealp{hardle1990applied}; \citealp{fan1992}; \citealp{gyorfi2002distribution}), implying the negligibility of the difference in a usual asymptotic framework (Figure \ref{fig:comparison}(\subref{subfig:conventional})). The situation drastically changes in nonparametric link regression as pairs of node covariates are dependent (despite outcomes at links being independent), and the difference cannot be ignored.

\begin{figure}[!ht]
\centering
\includegraphics[scale=0.22]{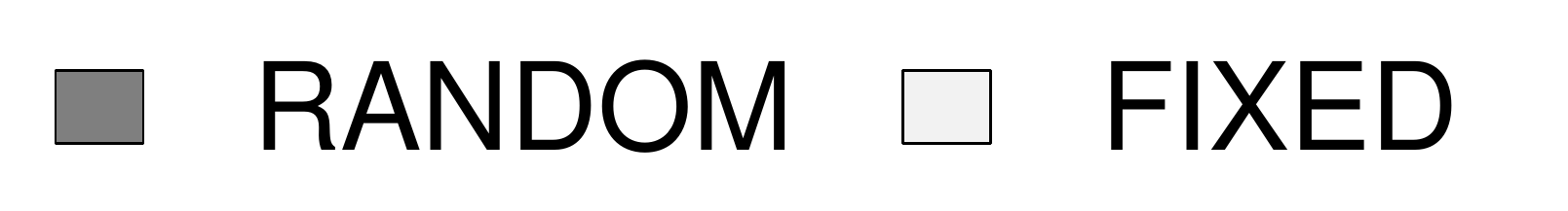} \\
  \begin{minipage}[t]{0.45\linewidth}
    \centering
    \includegraphics[scale=0.4]{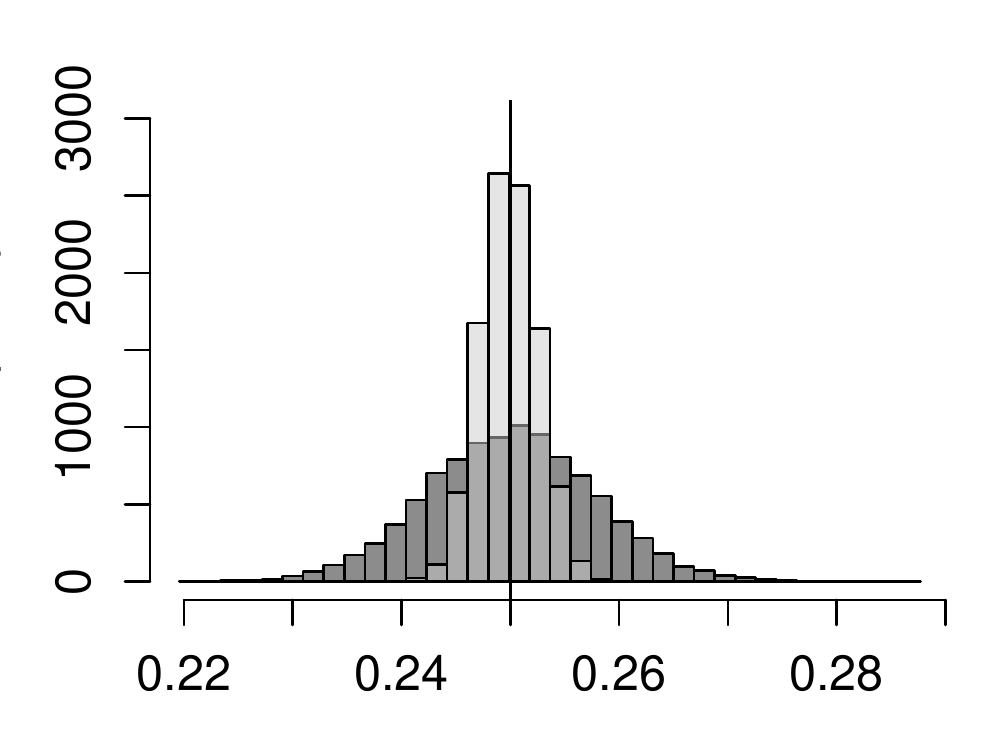}
    \subcaption[Nonparametric link regression]{\tabular[t]{@{}l@{}}Nonparametric link regression estimate \\ $\hat{f}_{n,h}(x,x')$ at a query $(x,x')=(0.5,0.5)$. \endtabular}
    \label{subfig:NLR}
\end{minipage}
  \begin{minipage}[t]{0.45\linewidth}
    \centering
    \includegraphics[scale=0.4]{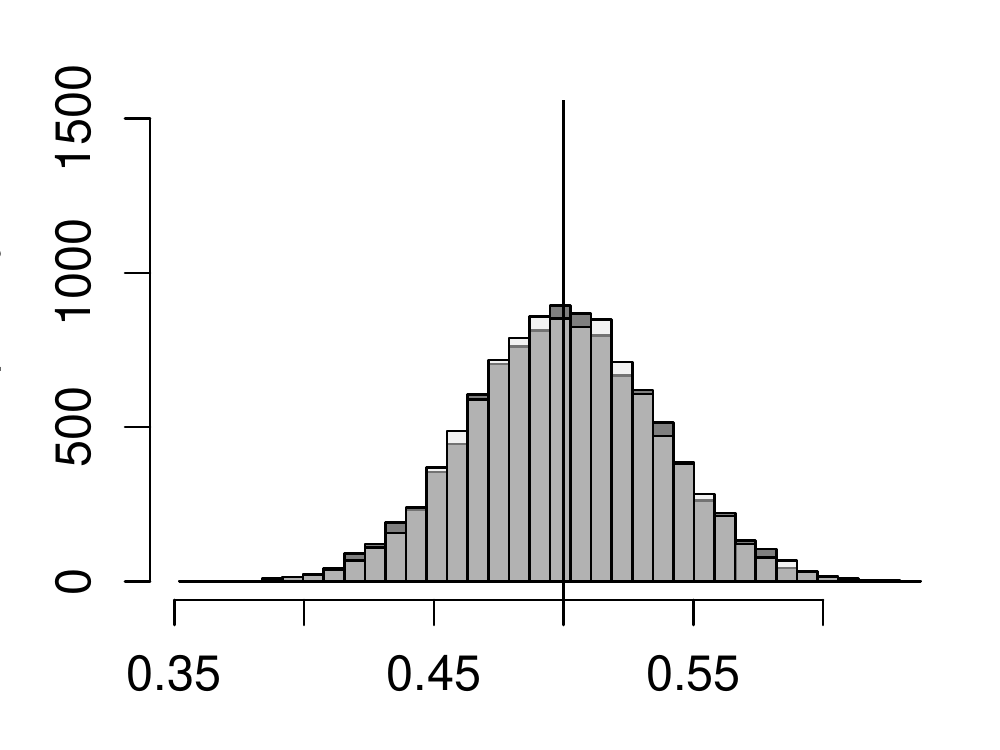}
    \subcaption[Nonparametric Reg.]{\tabular[t]{@{}l@{}}Nonparametric regression estimate \\ $\hat{f}_{n,h}(x)$ at a query $x=0.5$. \endtabular}
    \label{subfig:conventional}
  \end{minipage}
  \caption{
  Histograms of the values of estimates under univariate random and fixed designs: (\subref{subfig:NLR}) a kernel smoother for $f(x,x')=\mathbb{E}(Y \mid X=x,X'=x')=xx'$ defined in (\ref{eq: kernel smoother}). 
  (\subref{subfig:conventional}) a kernel smoother for $f(x)=\mathbb{E}(Y \mid X=x)=x$. For both experiments, 500 covariates were employed and the bandwidth was set to $h=h_n:=n^{-1/(d+3)}$ as the H{\"o}lder smoothness of $f$ is $\beta=1$.}
  \label{fig:comparison}
\end{figure}

\subsection{Notations}
The following notation is adopted throughout this paper. Let $\mathcal{X}$ be a compact subset of $\mathbb{R}^d$ and $\mathcal{Y}$ be a subset of $\mathbb{R}$. Let $\|\cdot\|$ denote the Euclidean norm of $\mathbb{R}^d$. For two sequences $\{a_{n,h}\}_{n \in \mathbb{N},h>0}$ and $\{b_{n,h}\}_{n \in \mathbb{N},h>0}$ the notation $a_{n,h}=O(b_{n,h})$ indicates that there exists an absolute constant $c>0$ for which $a_{n,h} \le c \cdot b_{n,h}$. Operations $\mathbb{E}[\cdot]$ and $\Var[\cdot]$ denote taking expectation and variance, respectively. For a conditional expression $\mathcal{S}$, $\mathbbm{1}_{\mathcal{S}}$ denotes an indicator function taking the value of 1 if and only if $\mathcal{S}$ is satisfied and 0 otherwise. For a set $A$, $|A|$ denotes the number of elements in the set.

\section{Nonparametric link regression}
\label{sec:nonparametric_link_regression}
Suppose that given node covariates $\{X_{i} \in \mathcal{X}:i=1,\ldots,n\}$, symmetric outcomes $Y_{i_{1},i_{2}}=Y_{i_2,i_1} \in \mathcal{Y}$ independently follow a conditional distribution $Q( Y_{i_{1},i_{2}} \mid X_{i_{1}}, X_{i_{2}})$ for $1 \le i_1 < i_2 \le n$. 
Our aim is to estimate the symmetric  conditional mean of $Y \mid X,X' \sim Q$ at a query $(x,x')$
\begin{align*}
f(x, x'):= \mathbb{E}[ Y \mid X=x\,,\,X'=x']
\end{align*} 
on the bases of the observations $\{(Y_{i_{1},i_{2}},X_{i_{1}},X_{i_{2}}) \in \mathcal{Y} \times \mathcal{X}^{2} : (i_{1},i_{2}) \in I_n \}$ with an index set $I_n:=\{(i_{1},i_{2}) : 1 \le i_{1} \ne i_{2} \le n\}$. 
In this study, the symmetric conditional mean is called a link regression function.
Consider nonparametric estimation over the link regression function. In this study, the following kernel smoother is used for the link regression function $f$:
\begin{align}
    \hat{f}_{n,h}(x,x') 
    &:= 
    \frac{|I_n|^{-1}\sum_{(i_1,i_2) \in I_n} Y_{i_1 i_2} \{ K_{h}(x-X_{i_1} )K_{h}( x' -X_{i_2}) +\lambda_n\}}
    {|I_n|^{-1}\sum_{(i_1,i_2) \in I_n} \{K_{h}(  x-X_{i_1} )K_{h}( x' -X_{i_2})+\lambda_n\}},
    \label{eq: kernel smoother}
\end{align}
where $K_{h}(\cdot):= h^{-d}K(\cdot/h)$ with a bandwidth $h>0$ and a $d$-variate kernel $K$, and $\lambda_{n}$ is a regularization parameter. 
Regularization is applied to the denominator as suggested by \citet{fan1993} to prevent it from becoming zero. Furthermore, the term $\lambda_n$ is added to the nominator for the kernel smoother to keep the form of the weighted average of outcomes. 

In our theory, the bandwidth $h=h_n$ and the regularization parameter $\lambda_n$ are assumed to satisfy
\begin{align}
    n^{-1/d} \le h \le 1, 
    \qquad 
    \frac{1}{(nh^d)^{\nu}} \le \lambda_n \le h^d
    \label{eq:bandwidth_regularization_conditions}
\end{align}
for some $\nu>0$; 
this condition yield inequalities $1 \le nh^d$ and $1/\{\lambda_n (nh^d)^{\nu}\} \le 1$, and the condition is satisfied by $\lambda_n=n^{-1}$ with the optimal bandwidth $h=n^{-1/(\beta+d)}$ obtained in Theorem~\ref{theorem: lower bound}. 

It is well-known that the kernel smoother has a limitation to attaining the optimal rate of convergence for function classes with higher order smoothness~(\citet{ruppert1994multivariate}); this limitation is alleviated by local polynomial regression. As the extension to local polynomial regression is straightforward in nonparametric link regression, it is not considered in this paper for the ease of notation.

\subsection{Theoretical result}

Properties of the kernel smoother $\hat{f}_{n,h}$ are derived, with the following conditions imposed on the true link regression function $f$,  conditional variance $\sigma$, kernel $K$ and the covariate $X$.

\begin{condition}\label{condition: link regression function} 
The following hold:
\begin{enumerate}
    \item[(a)] 
	The link regression function  $f(x,x')$ 
	is contained in $\mathcal{F}(\beta,L)$ with $0< \beta \le 1$ and $L>0$ defined as
    $\mathcal{F}(\beta,L)
        :=
        \left\{ f: 
        |f(x,x') - f(\tilde{x},x')| \le L\|x-\tilde{x}\|^{\beta}, \, 
        |f(x,x')| \le L
        \,,\,x,\tilde{x},x'\in \mathcal{X}
        \right\}$; 
    \item[(b)] There exists $\tau>0$ for which $\sigma^{2}(x,x'):=\Var [Y\mid X=x,\,X'=x']< \tau$ for $x,x'\in\mathcal{X}$.
\end{enumerate}
\end{condition}

\begin{condition}\label{condition: kernel} 
The following hold:
\begin{enumerate}
    \item[(a)] The $d$-variate kernel $K(z)$ is a compactly supported, symmetric, bounded density function with a mean of zero and $K_{\mathrm{max}}:=\sup_{z} K(z)$, and has the finite second moment;
    \item[(b)] There exist $\underk>0$ and $r>0$ for which $\underk \cdot \sup_{z}K(z) \le \inf_{z: \|z\| \le r} K(z)$.
\end{enumerate}
\end{condition}
\begin{condition}\label{condition: design}
One of the following holds:
\begin{enumerate}

	\item[(a)]
	$\{X_{i}:i=1,2,\ldots,n\}$ is a deterministic triangular array that satisfies $c_X / n^{1/d}\le \min_{i,j} \|X_{i}-X_{j}\| \le  \max_{i}\min_{j} \|X_{i}-X_{j}\|\le C_X/n^{1/d}$ for some $c_X,C_X>0$;
	
	\item[(b)] $\{X_{i}:i=1,2,\ldots,n\}$ is an array of independent and identically distributed (i.i.d.) random variables from the marginal density $m$, where $m$ is continuous and bounded from both above and below: $0< l \le m(x) \le u <\infty$.
\end{enumerate}
\end{condition}

Condition \ref{condition: link regression function} specifies the smoothness of the link regression function $f$. A larger $\beta$ yields a smoother function $f$, which is easier to estimate. Condition \ref{condition: kernel} is mild and satisfied by a large variety of kernels such as the boxcar and Epanechnikov kernels. Condition \ref{condition: design} (a) makes covariates nearly equispaced. Condition \ref{condition: design} (b) is a mild condition; for example, the uniform distribution over $\mathcal{X}$ satisfies this condition.  Hereafter, we call Conditions \ref{condition: design} (a) and \ref{condition: design} (b) as the fixed and random designs, respectively.
Let
\begin{align*}
    \rho_{1}(n,h) &:= \sup_{f \in \mathcal{F}(\beta,L)}
    \left|f(x,x')-\mathbb{E}\left[\hat{f}_{n,h}(x,x')\right]\right|^2, \\
    \rho_{2}(n,h) &:= \sup_{f \in \mathcal{F}(\beta,L)}\mathbb{E}\left[\Var\left[\hat{f}_{n,h}(x,x') \mid X_1,X_2,\ldots,X_n\right]\right], \\
    \rho_{3}(n,h) &:= \sup_{f \in \mathcal{F}(\beta,L)}\Var\left[
    f(x,x')-
    \mathbb{E}\left[\hat{f}_{n,h}(x,x') \mid X_1,X_2,\ldots,X_n\right]\right]
\end{align*}
where these terms upper-bound the squared bias and the variance of $\hat{f}_{n,h}(x,x')$ by $\rho_1$ and $\rho_2+\rho_3$, respectively, whereby the overall risk $\mathbb{E}|f(x,x')-\hat{f}_{n,h}(x,x')|^2$ is upper-bounded by $\rho_1+\rho_2+\rho_3$. 
These terms $\rho_1,\rho_2,\rho_3$ are evaluated as follows, with the proof provided in Supplement~\ref{app:proof of proposition: upper bounds of terms}. 
\begin{theorem}\label{proposition: upper bounds of terms}
Assume that 
Conditions \ref{condition: link regression function} and \ref{condition: kernel} are satisfied, and
$n^{-1/d} \le h \le 1$ and $1/(nh^d)^{\nu} \le \lambda_{n}\le h^{d}$ for some $\nu>0$. Then, there exists a positive constant $C$ not depending on $n$ and $h$, for which the following holds:
\begin{itemize}
    \item In both cases of Condition \ref{condition: design}, $\rho_{1}(n,h) \le C h^{2\beta}$ and $\rho_{2}(n,h) \le C (nh^{d})^{-2}$;
    \item Under Condition \ref{condition: design} (a), $\rho_{3}(n,h) =0$; under Condition \ref{condition: design} (b), $\rho_{3}(n,h) \le C h^{2\beta}(nh^{d})^{-1}$.
\end{itemize}
\end{theorem}

With the bandwidth $h=n^{-1/(s+d)}$, Theorem~\ref{proposition: upper bounds of terms} yields the following variance evaluation: 
\begin{align*}
\Var[ \hat{f}_{n,h}(x,x') ]
=
\begin{cases}
O\left( n^{-\frac{2s}
{s+d}}
\right) &
\text{ under Condition \ref{condition: design} (a), } \\
O\left( n^{-\frac{
\min\{
2s,2\beta+s
\}
}{s+d}} \right) &
\text{ under Condition \ref{condition: design} (b). }
\end{cases}
\end{align*}

This implies that the stochastic variance in the random design decays slower than that in the fixed design for $s>2\beta$. This also implies that the decay rates of the stochastic variances in both designs match in the other regimes. Hence, the smoothness misspecification produces a difference in variances with respect to the designs. 
See also Supplement~\ref{sec:numerical_studies} for numerical experiments demonstrating the dependence of the variance on the covariate design.

Theorem~\ref{proposition: upper bounds of terms} also yields the upper-bound of the risk. 

\begin{theorem}\label{theorem: upper bound}
Assume that Conditions \ref{condition: link regression function} and \ref{condition: kernel} are satisfied, $n^{-1/d} \le h \le 1$ and $1/(nh^d)^{\nu} \le \lambda_{n}\le h^{d}$ for some $\nu>0$.
Then, there exists a positive constant $C$ not depending on $n$ and $h$, for which the following holds:
\begin{align*}
&\sup_{f \in \mathcal{F}(\beta,L)}\mathbb{E}|f(x,x')-\hat{f}_{n,h}(x,x')|^2 
 \le
    \begin{cases}
     C\{h^{2\beta} + (nh^{d})^{-2}\}
     & \text{under Condition \ref{condition: design} (a)},\\
     C \{h^{2\beta}+
     h^{2\beta}(nh^{d})^{-1}+ (nh^{d})^{-2}\}
     & \text{under Condition \ref{condition: design} (b)}.
    \end{cases}
\end{align*}
\end{theorem}

Combined with the lower bound estimate of the risks, the above evaluation also yields the optimal bandwidth. If the bandwidth $h$ is set to $h=n^{-1/(\beta+d)}$ with a true smoothness $\beta$, the optimal rate of convergence is obtained in both designs.

\begin{theorem}\label{theorem: lower bound}
Assume that the conditional distribution $Q$ is a normal distribution. In either Condition \ref{condition: design} (a) or (b), the kernel smoother with the bandwidth $h= n^{-1/(\beta+d)}$ attains asymptotically minimax optimality
\begin{align*}
    \limsup_{n\to\infty} \bigg\{ 
    \sup_{f \in \mathcal{F}(\beta,L)}\mathbb{E}\left[|f(x,x')-\hat{f}_{n,h}(x,x')|^2\right] 
 \big{/} 
    \inf_{\hat{f}_n}\sup_{f\in \mathcal{F}(\beta , L)}\mathbb{E}\left[ |f(x,x')-\hat{f}_n(x,x')|^{2} \right] 
\bigg\}<\infty,
\end{align*}
and the asymptotically minimax risk is of order $n^{-2\beta/(\beta+d)}$ with respect to $n$.
\end{theorem}
The proof is presented in 
Supplement~\ref{app:proof of lower}. 
The theorem also holds for cases where the conditional distribution satisfies the conditions in Section 2 of \citet{GillandLevit(1995)}; for example, the theorem holds for cases where $Q$ is a Bernoulli distribution and $\mathcal{F}(\beta,L)$ is replaced by $\mathcal{F}(\beta,L)\cap \{f:0<l<\inf f(x,x') \le \sup f(x,x')<u<1\}$ with $0<l<u<1$.

Note that, a similar but different result on the variance of the link regression under the random design has been reported recently~\citep{graham2021minimax}, independently to this work. 
\citet{graham2021minimax} employs the probabilistic model $Y_{ij}=f(X_i,X_j)+U_i+U_j+V_{ij}$ with i.i.d.~normal random variables $U_i,V_{ij} \sim N(0,1)$ while our setting corresponds to $Y_{ij}=f(X_i,X_j)+V_{ij}$ if $Q$ is a normal distribution. The existence of the node-dependent term $U_i$ yields their asymptotically optimal minimax risk $n^{-2\beta/(2\beta+d)}$ different from ours $n^{-2\beta/(\beta+d)}$. See Supplement~\ref{app:graham} for more detailed comparison to \citet{graham2021minimax}.

\section{Conclusion}
We have discussed nonparametric link regression, demonstrating that the asymptotic variance decay rate of nonparametric link regression estimates depends on the covariate design; namely, whether the design is random or fixed.

\section*{Acknowledgement}
We would like to thank Hidetoshi Shimodaira for helpful discussions
and thank two anonymous referees for carefully reviewing the paper and giving us constructive comments. 
A. Okuno is supported by JST CREST (JPMJCR21N3) and JSPS KAKENHI (21K17718). 
K. Yano is supported by JST CREST (JPMJCR1763), JSPS KAKENHI (19K20222, 21H05205, 21K12067), and MEXT (JPJ010217).

\clearpage

\appendixpageoff
\appendixtitleoff
\renewcommand{\appendixtocname}{Supplementary material}
\begin{appendices}

\noindent {\Large \textbf{Supplementary material:}}
\begin{center}
{\Large Dependence of variance on covariate design  in nonparametric link regression} \\
\vspace{0.5em}
Akifumi Okuno and Keisuke Yano
\end{center}

This supplementary material contains the comparison to \citet{graham2021minimax}, numerical experiments, and proofs of main results, 
in Supplement~\ref{app:graham}--\ref{app:proofs}.

\section{Comparison to \citet{graham2021minimax}}
\label{app:graham}

Our work and \citet{graham2021minimax} are different in the following points:

\begin{enumerate}
    \item \textbf{Probabilistic model.} \citet{graham2021minimax} employs the probabilistic model $Y_{ij}=f(X_i,X_j)+U_i+U_j+V_{ij}$ with $U_i,V_{ij} \sim N(0,1)$, while we consider $Y_{ij} \sim Q(f(X_i,X_j))$; our model reduces to $Y_{ij}=f(X_i,X_j)+V_{ij}$ by assuming that $Q$ is a normal distribution. 
    The assumption on the existence of additive node-dependent noises depends on target real-data and makes the behaviour of the estimates quite different.
    
    \item \textbf{Kernel.} \citet{graham2021minimax} employs a kernel $\tilde{K}(z-Z)$ for the concatenated vectors $z=(x,x')$ and $Z=(X,X')$ while we employ a product kernel $K(x-X)K(x'-X')$ to enjoy the symmetry. 
    
    \item \textbf{Kernel assumptions.} \citet{graham2021minimax} further imposes the Lipschitz property $|\tilde{K}(Z)-\tilde{K}(Z')| \le L\|Z-Z'\|$, differentiablity, and the polynomial decay of the derivative $\|\partial \tilde{K}(Z)/\partial Z\| \le L' \|Z\|^{-\nu}$ for some $\nu>1$, while we impose only the boundedness, and the compactness of the support, on the kernel $K$. 
    
    \item \textbf{Regularization.} \citet{graham2021minimax} does not introduce any regularization unlike ours. 
    The regularization is needed for random design analysis, so as to prevent the denominator of the kernel smoother from being $0$.  
\end{enumerate}

Due to the above differences, evaluation of the asymptotically optimal minimax risk in \citet{graham2021minimax} is different from  ours:
\[
    \text{Graham et al.~'s:} \: n^{-2\beta/(2\beta+d)},
    \quad 
    \text{Ours:} \: n^{-2\beta/(\beta+d)}. 
\]
The different convergence rate is due to the variance. 
With our decomposition
\[
    \mathbb{V}(\hat{f}_{n,h}(x,x'))
    \le 
    \underbrace{\sup_{f \in \mathcal{F}(\beta,L)}
    \mathbb{E}(\mathbb{V}(\hat{f}_{n,h}(x,x') \mid X))}_{=\rho_2}
    +
    \underbrace{
    \sup_{f \in \mathcal{F}(\beta,L)}
    \mathbb{V}(\mathbb{E}(\hat{f}_{n,h}(x,x') \mid X))}_{=\rho_3},
\]
the term $\rho_2$ is evaluated as $\rho_2=O((nh^d)^{-2})$ in our setting (as outcomes $\{Y_{ij}\}$ are $\binom{n}{2}=O(n^2)$ independent random variables). 
However, $\rho_2=O((nh^d)^{-1})$ in Graham's setting, as $Y_{ij},Y_{ik}$ are dependent (as they share $U_i$ therein) and thus the variance is $O((nh^d)^{-1})$ by following the same calculation as $U$-statistic. Thus the term $O((nh^d)^{-1})$ is dominant in \citet{graham2021minimax}, and it determines the overall evaluation of the risk. 
Note that the above evaluation is tight: both of this study and \citet{graham2021minimax} attain minimax optimality in each setting.

\section{Numerical studies}
\label{sec:numerical_studies}

Numerical experiments are employed to examine the kernel smoother~(\ref{eq: kernel smoother}). 
For the univariate case ($d=1$) with a true link regression function given by $f(x,x')=xx'$, node covariates $\{x_i \in [0,1] \,:\, i=1,2,\ldots,n\}$ and outcomes $\{y_{i_1 i_2} \in \{0,1\} \,:\, 1 \le i_1 \ne i_2 \le n\}$ are synthetically generated in the following way:

\begin{itemize}
\item In the fixed design, $x_i=(i-1)/(n-1) \in [0,1]$ for $i=1,2,\ldots,n$; and in the random design, i.i.d. node covariates $\{x_{i}\,:\,i=1,\ldots,n\}$ are generated from a uniform distribution over $[0,1]$. 
\item Outcomes $\{y_{i_1 i_2} \,:\, 1 \le i_{1} < i_{2} \le n\}$ are generated independently from the Bernoulli distribution with the mean $f(x_{i_1},x_{i_2}) \in [0,1]$ and $y_{i_1 i_2}:=y_{i_2 i_1}$ for $1 \le i_2 < i_1 \le n$.
\end{itemize}

Consider the kernel smoother~(\ref{eq: kernel smoother}) equipped with the boxcar kernel $K(x)=\mathbbm{1}_{\|x\|_2 \le 1}$ and bandwidth $h=h_n:=n^{-1/(s+d)}$, where $d=1$, and $s \in \{0.75,1,2,3\}$. 
A regularization parameter $\lambda_n=n^{-1}$ is employed. 
Note that the true link regression function $f(x,x')=xx'$ is contained in $\mathcal{F}(1,1)$.

The histograms of the values of $\hat{f}_{n,h}(x,x')$ are calculated at a fixed query $(x,x')=(0.5,0.5)$ in the random and fixed designs using $10^4$ synthetic datasets. 
Figure~\ref{table:histograms} summarizes the results. It is noticed from the table that variances depend on the covariate design and the difference in histograms becomes smaller as $s$ decreases. 

\begin{figure}[!ht]
    \centering
    \includegraphics[scale=0.22]{legend.pdf} \\
\scalebox{0.9}{
    \begin{tabular}{M{33mm}M{33mm}M{33mm}M{33mm}}
        $s=0.75$ & $s=1$ & $s=2$ & $s=3$ \\ 
        \includegraphics[scale=0.4]{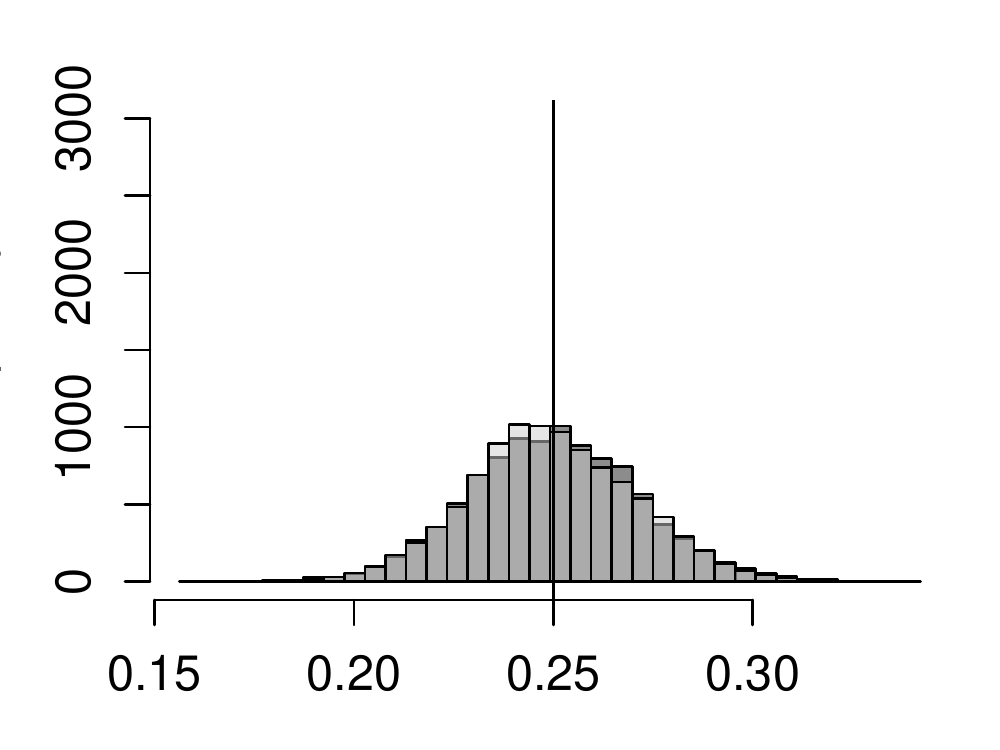} &
        \includegraphics[scale=0.4]{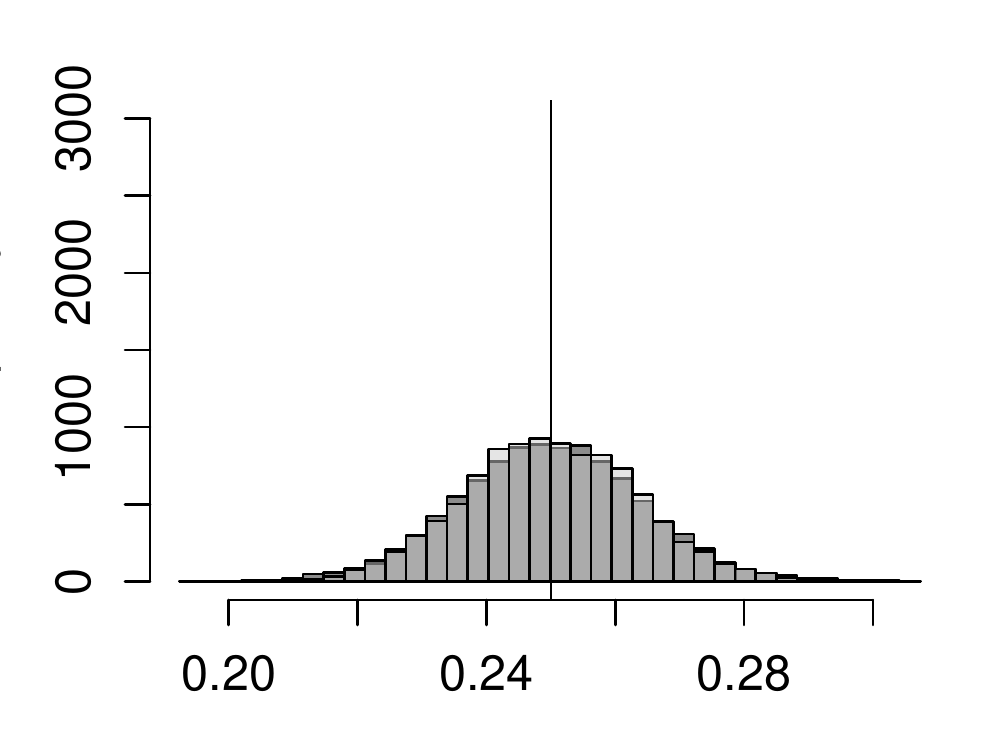} &
        \includegraphics[scale=0.4]{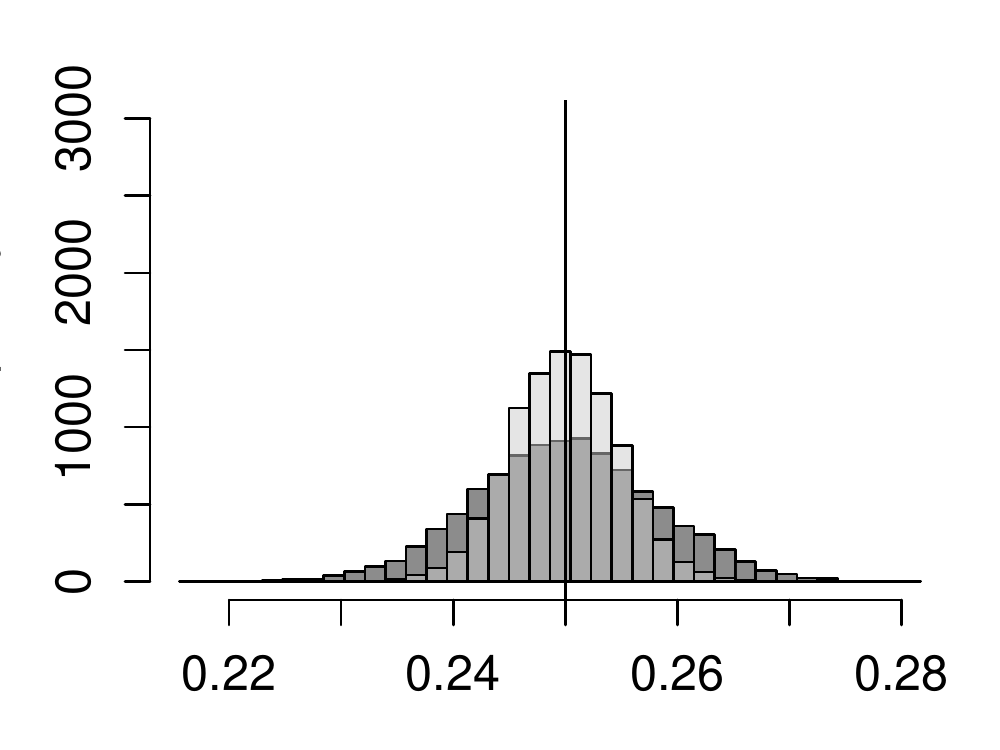} &
        \includegraphics[scale=0.4]{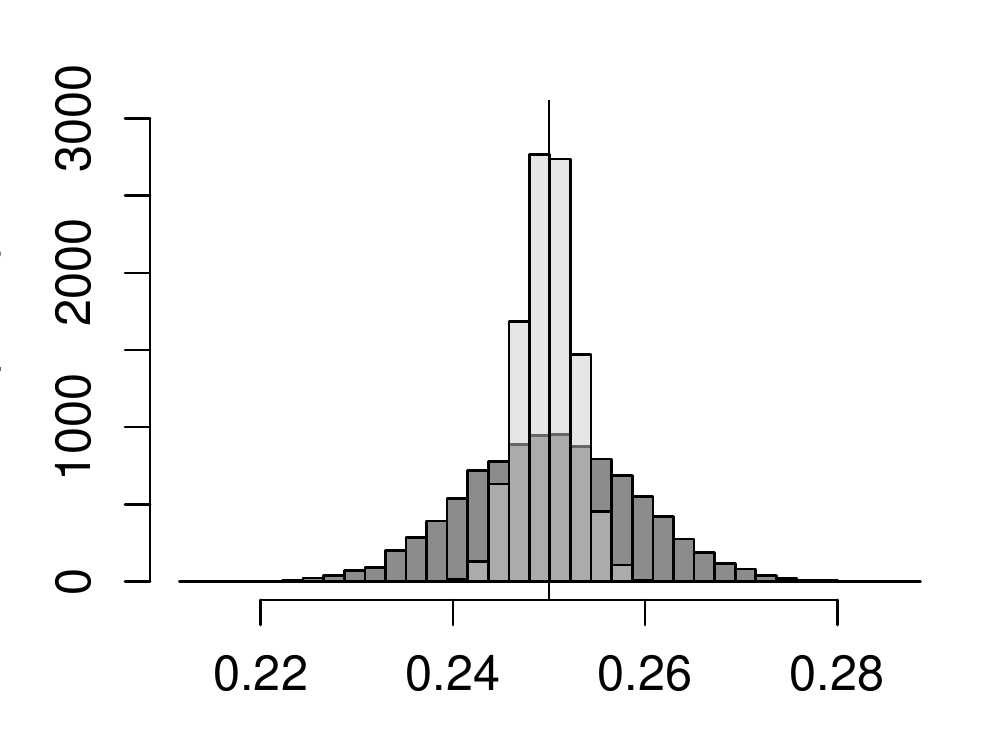} \\
    \end{tabular}
}
    \caption{Histograms of the kernel smoothers~(\ref{eq: kernel smoother}) at the query $(x,x')=(0.5,0.5)$. The bandwidth is set to $h=n^{-1/(s+d)}$ with respect to the number of nodes $n=500$. The black vertical line represents the true conditional expectation $f(0.5,0.5)=0.25$. 
    }
    \label{table:histograms}
\end{figure}

\begin{figure}[!ht]
    \centering
    \includegraphics[scale=0.22]{legend.pdf} \\
    \scalebox{0.9}{
    \begin{tabular}{M{33mm}M{33mm}M{33mm}M{33mm}}
        $s=0.75$ & $s=1$ & $s=2$ & $s=3$ \\ 
        \includegraphics[scale=0.33]{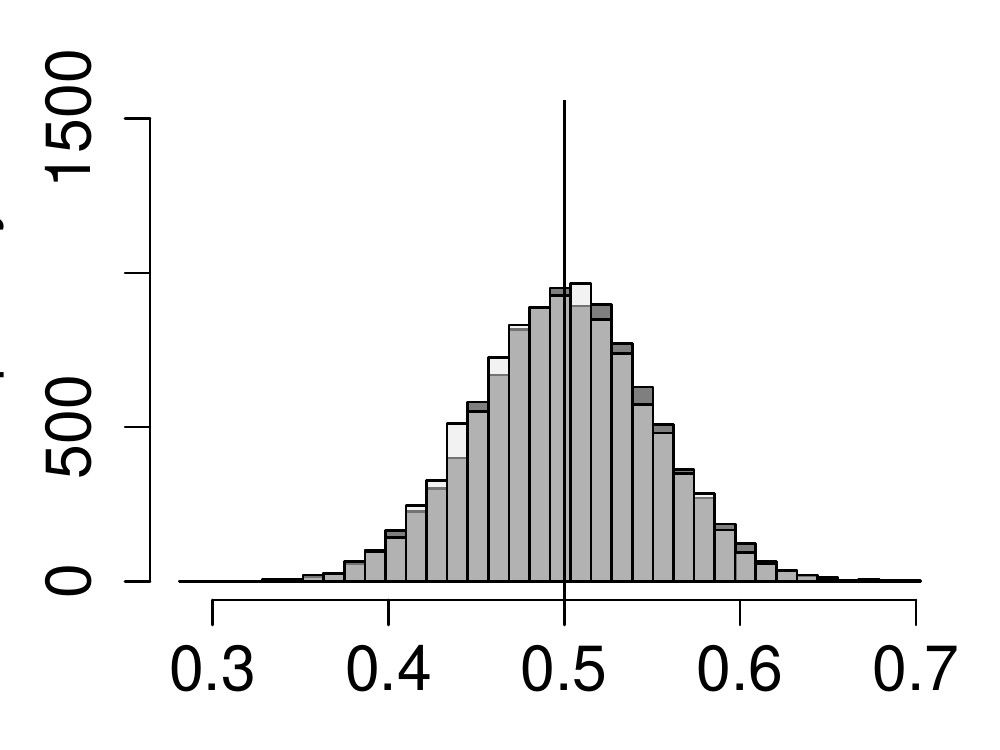} &
        \includegraphics[scale=0.33]{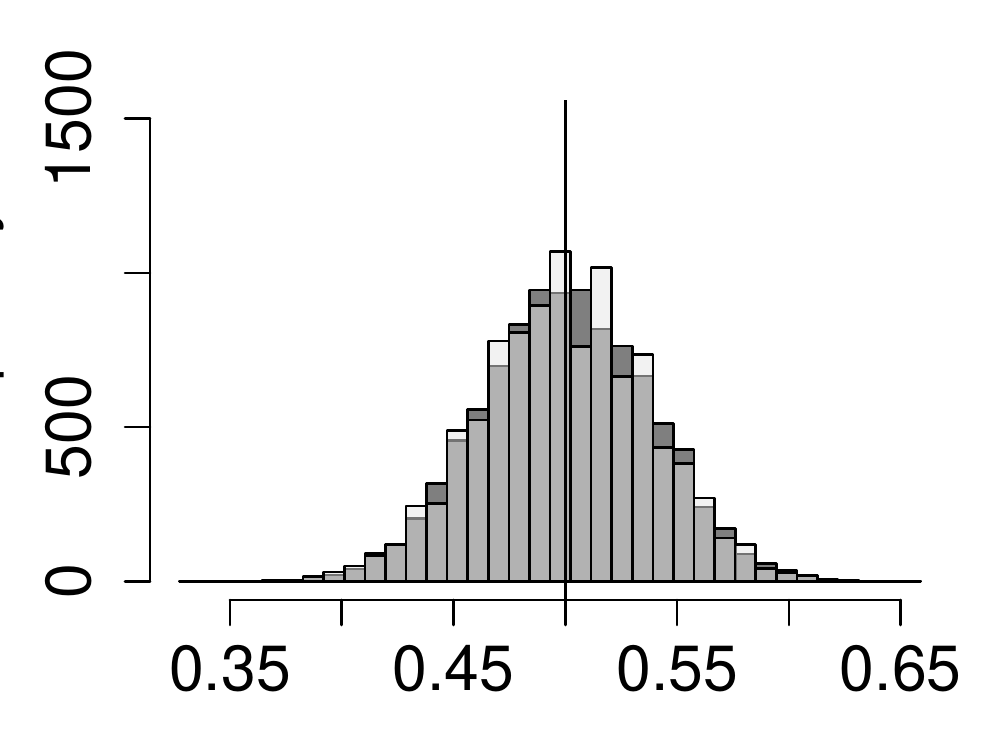} &
        \includegraphics[scale=0.33]{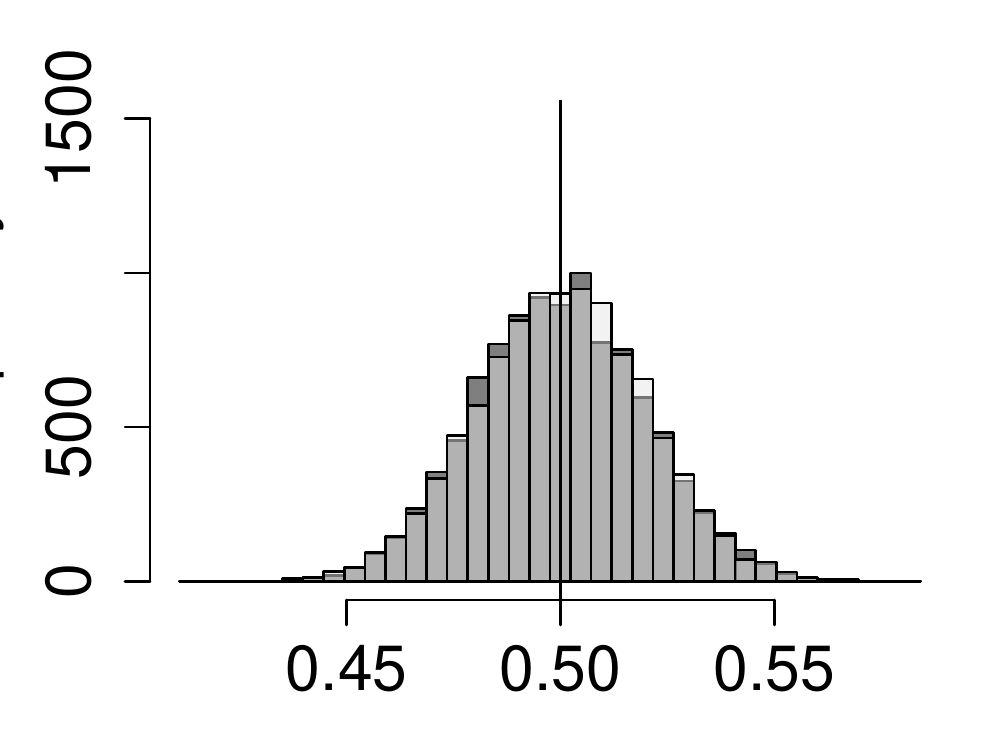} &
        \includegraphics[scale=0.33]{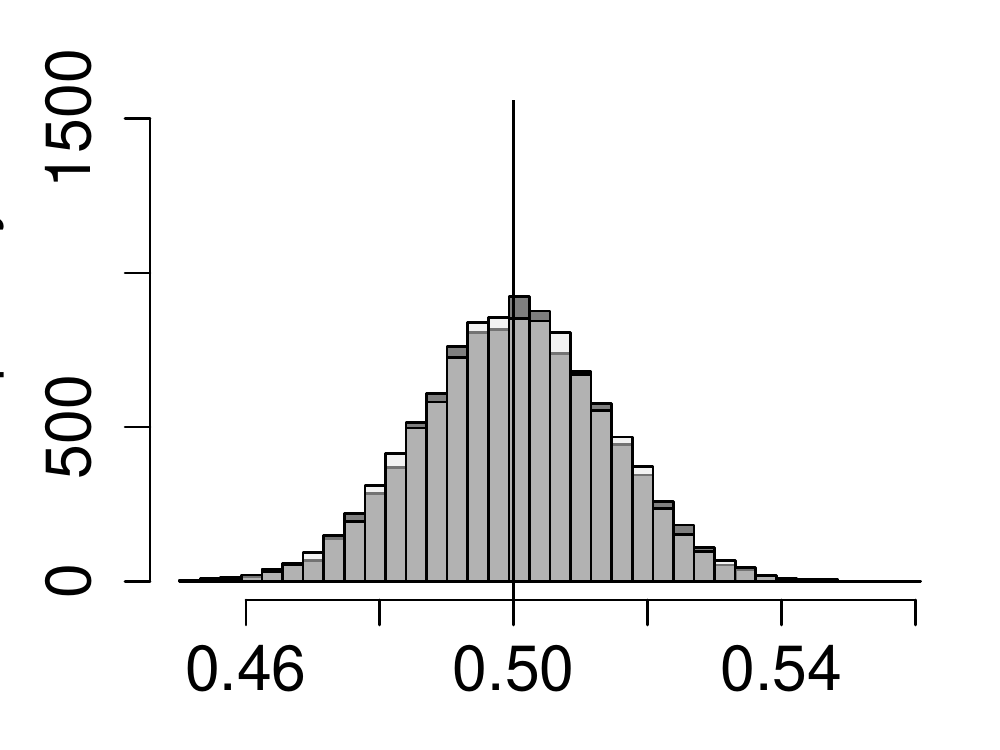} \\
    \end{tabular}
    }
    \caption{Histograms of the kernel smoother~(\ref{eq: conventional nonparametric regression}) at the query $x=0.5$. The bandwidth and the regularization parameter are $h=n^{-1/(s+d)}$, $\lambda_n=n^{-1/2}$, respectively, where $n=5000$ represents the number of covariates.}
    \label{table:histograms_conventional}
\end{figure}

Herein, we present similar experiments for nonparametric regression. 
Given covariates $\{X_i \in \mathcal{X}:i=1,\ldots,n\}$, an outcome $Y_i \in \mathcal{Y}$ independently follows a conditional distribution $Q(Y_i \mid X_i)$. For estimating the conditional mean $f(x):=\mathbb{E}[Y \mid X=x]$ at a query $x$, we define a kernel smoother
\begin{align}
    \hat{f}_{n,h}(x)
    :=
    \frac{
        n^{-1}\sum_{i=1}^{n} Y_i \{K_h(x-X_{i})+\lambda_n\}
    }{
        n^{-1}\sum_{i=1}^{n} \{K_h(x-X_{i})+\lambda_n\}
    }
    \label{eq: conventional nonparametric regression}
\end{align}
similarly to the nonparametric link regression~(\ref{eq: kernel smoother}).

Using the function $f(x)=x$ and a query $x=0.5$, the histograms of the estimator (\ref{eq: conventional nonparametric regression}) via $10^4$ times experiments are illustrated in the following Figure~\ref{table:histograms_conventional}.

\section{Proofs of main results}
\label{app:proofs}

Using supporting Lemmas (shown in Supplement~\ref{app:supporting_lemmas}), this section provides the proofs of Theorem~\ref{proposition: upper bounds of terms} (shown in Supplement~\ref{app:proof of proposition: upper bounds of terms}) and Theorem~\ref{theorem: lower bound} (shown in Supplement~\ref{app:proof of lower}). 
For proofs, we employ the following symbols:
\begin{align*}
W_{i_1,i_2}(x,x';X)
&:=
\frac{
    K_{h}(x-X_{i_{1}})K_{h}(x'-X_{i_{2}}) +\lambda_n
}{
    \sum_{(i_1,i_2) \in I_n}\{K_{h}(x-X_{i_{1}})K_{h}(x'-X_{i_{2}})+\lambda_n\}
}, \\
S_{n,h}&:=|I_n|^{-1}\sum_{(i_1,i_2) \in I_n}
\{ f(X_{i_1},X_{i_2}) - f(x,x') \}
\{ K_h(x-X_{i_1})K_h(x'-X_{i_2}) 
+\lambda_n \} , \\
T_{n,h}&:=|I_n|^{-1}\sum_{(i_1,i_2) \in I_n}\{K_h(x-X_{i_1})K_h(x'-X_{i_2})+\lambda_n\}, \\
\overline{T}_{n,h}&=\mathbb{E}(T_{n,h}), \quad \text{and }
\varepsilon_{n,h}:=\frac{1}{T_{n,h}}-\frac{1}{\overline{T}_{n,h}}.
\end{align*}
These satisfy
\begin{align}
\hat{f}_{n,h}(x,x')
&=
\sum_{(i_1,i_2) \in I_n}Y_{i_1i_2}W_{i_1,i_2}(x,x';X),
\label{eq: sum expression of f hat} \\
\sum_{(i_1,i_2) \in I_n}W_{i_1,i_2}(x,x';X)&=1, \text{ and } 
\label{eq: sum of W}  \\
\mathbb{E}(\hat{f}_{n,h}(x,x') \mid X) - f(x,x') &= \frac{S_{n,h}}{T_{n,h}}.
\label{eq: frac expression}
\end{align}

\subsection{Supporting lemmas}
\label{app:supporting_lemmas}

We begin with stating supporting lemmas used in the proofs of main results. 
Expression of variance of $U$-statistics (Lemma \ref{lemma: variance of u}); upper estimates of moments of the kernel (Lemma \ref{lemma: common bound}); the Rosenthal inequality for $U$-statistics (Lemma \ref{lemma: Rosenthal}); the small ball estimate (Lemma \ref{lemma: small ball}); 
a lower bound of the number of observed covariates in a ball (Lemma \ref{lemma: number of positive points}); 
upper bounds of moments of $T_{n,h}$ (Lemma \ref{lemma: moment of T}); 
upper bounds of moments of $\varepsilon_{n,h}$ (Lemma \ref{lemma: moment of epsilon}); An upper bound of the variance of $S_{n,h}$ (Lemma \ref{lemma:var_snh}); 
the fourth moment evaluation of $S_{n,h}$ (Lemma \ref{lemma:var_snhenh}).

For a bounded function $k:\mathcal{X}^{2}\to\R$ and an i.i.d.~sequence $\{X_{1},\ldots,X_{n}\}$, let 
\[
U_{n} := 
|I_n|^{-1}
\sum_{(i_{1},i_{2}) \in I_n} k(X_{i_1},X_{i_{2}})
=
\binom{n}{2}^{-1}\sum_{1 \le i_1 < i_2 \le n}\left\{
    \frac{k(X_{i_1},X_{i_2})+k(X_{i_2},X_{i_1})}{2}
\right\}.
\]

From the symmetric expression of $U_{n}$ (the rightmost  in the definition of $U_{n}$), we get the following expression on the variance of $U_{n}$.
\begin{lemma}[\citealp{lee1990u}]
\label{lemma: variance of u}
We have
\[\Var[U_n]
=
\frac{4(n-2)}{n(n-1)}\zeta_1
+
\frac{2}{n(n-1)}\zeta_2,
\]
where 
\begin{align*}
    \zeta_1:=\Var\left[\mathbb{E}\left[\frac{k(X,X')+k(X',X)}{2} \mid X\right]\right]\text{ and }
\zeta_2:=\Var\left[\frac{k(X,X')+k(X',X)}{2}\right].
\end{align*} 
In addition, we get 
\begin{align*}
    \zeta_1 &\le \frac{1}{2}\left[
    \Var[\mathbb{E}[k(X,X') \mid X]]
    +
    \Var[\mathbb{E}[k(X',X) \mid X]]
\right], \text{ and }\\
\zeta_2 &\le 
\Var[k(X,X')].
\end{align*}
\end{lemma}

We need an upper bound on the expectation of $K_{h}^{\alpha}(x-X)\norm{x-X}^{\beta}$
for $\alpha,\beta \ge 0$.
Let
\[
    \kappa_{\alpha,\beta}:=
    \|m\|_{\infty}
    \left\{
    \int_{\mathbb{R}^d} K(z)^{\alpha} \diff z 
    \right\}
    \left\{
    \sup_{z \,\in\, \text{supp}\, K} \norm{z}^{\beta}
    \right\} \quad (\text{with }0^0=0 \text{ for }\alpha=0),
\]
which is bounded since $K$ is compactly supported. 
We obtain the following Lemma~\ref{lemma: common bound}.
\begin{lemma}\label{lemma: common bound}
We have
$\mathbb{E}[K_h(x-X)^{\alpha} \norm{x-X}^{\beta} ]\le \kappa_{\alpha,\beta} h^{(1-\alpha)d+\beta}$ for any $\alpha, \beta \ge 0$. 
\end{lemma}

\begin{proof}
Letting $Z:=(x-X)/h$, we have 
\begin{align*}
    \mathbb{E}[K_h(x-X)^{\alpha}\norm{x-X}^{\beta}] 
    &=
    \int_{\mathcal{X}} K_h(x-X)^{\alpha}\norm{x-X}^{\beta}m(X) \diff X \\
    &=
    \int \{h^{-d}K(Z)\}^{\alpha}\norm{hZ}^{\beta} m(x+hZ) h^d\diff Z \\
    &\le 
    \|m\|_{\infty}
    \left\{\int K(Z)^{\alpha} \diff Z\right\}
    \left\{\sup_{Z \,\in\, \text{supp}\,K}\norm{Z}^{\beta}\right\}
    h^{(1-\alpha) d+\beta} \\
    &=
    \kappa_{\alpha,\beta} h^{(1-\alpha) d+\beta}. 
\end{align*}
\end{proof}

We use the following Rosenthal type estimate of the fourth moment of $U$-statistics.
\begin{lemma}\label{lemma: Rosenthal}
There exists an absolute constant $C_{R}>0$ for which we have
\begin{align*}
    &\mathbb{E}\left[\left(
    \sum_{(i_{1},i_{2})\in I_{n}}k(X_{i_{1}},X_{i_{2}})\right)^{4}\right]
    \\
    &\le 
    C_{R}
    \Bigg\{
    \Bigg(\sum_{(i_{1},i_{2})\in I_{n}}
    \mathbb{E}[k(X^{(1)}_{i_{1}},X^{(2)}_{i_{2}})]
    \Bigg)^{4}
    +
    \sum_{1 \le i_{1} \le n}\mathbb{E}\left[\Bigg( \sum_{ 1 \le i_{2} \le n} \mathbb{E}[ k(X_{i_{1}}^{(1)},X_{i_{2}}^{(2)} ) \mid X_{i_{1}}^{(1)}] \Bigg)^{4}\right]
    \\
    &\quad\quad\quad +
    \sum_{1 \le i_{2} \le n}\mathbb{E}\left[\Bigg( \sum_{1 \le i_{1} \le n} \mathbb{E}[ k(X_{i_{1}}^{(1)},X_{i_{2}}^{(2)} ) \mid X_{i_{2}}^{(2)}] \Bigg)^{4}\right]
    +
    \sum_{(i_{1},i_{2})\in I_{n}}
    \mathbb{E}[\{k(X^{(1)}_{i_{1}},X^{(2)}_{i_{2}})\}^{4}]
    \Bigg\}.
\end{align*}
\end{lemma}
\begin{proof}
Combining the decoupling inequality in \cite{delaPenaandMontromery-Smith}
and the Rosenthal inequality (2.2'') in  \cite{GineLatalaZinn(2000)}
gives the desired inequality.
\end{proof}

We also use the following lower estimate of a small ball probability.

 \begin{lemma}
 \label{lemma: small ball}
 Let $X$ be a random vector from a density $m$ satisfying Condition \ref{condition: design} (b). 
 For any $x_{0}\in\mathcal{X}$ and any $h>0$, we have
 \begin{align*}
 \Pr( \|X-x_{0}\|\le h ) &\ge
 \frac{ l r^{d}\pi^{d/2}}{\Gamma(d/2+1)}
 h^{d} =: C_{B}h^{d}, \\
 \Pr( \|X-x_{0}\|\le h ) &\le
 \frac{ u r^{d}\pi^{d/2}}{\Gamma(d/2+1)}
 h^{d} =: C'_{B}h^{d},
 \end{align*}
where $\Gamma(x)$ is the Gamma function.
\end{lemma}
The proof is easy and omitted.

In the following, the number of observed covariates considered in the nonparametric link regression is evaluated. 

\begin{lemma}
\label{lemma: number of positive points}
Under the condition~\ref{condition: design} (a) with $n^{-1/d} \le h \le 1$, 
there exists $C_K>0$ independent of $n,h$ such that
\[
    \frac{1}{|I_n|} \sum_{(i_1,i_2) \in I_n} K_h(x-X_{i_1})K_h(x'-X_{i_2}) \ge C_K.
\]
\end{lemma}

\begin{proof}
Let $r:=\sup\{\|x\|_2 \mid K(x)>0\}$, and let 
\[
    \mathcal{P}_{n,h}(x):=\left\{i \in [n] \mid \left\| \frac{x-X_i}{h} \right\| \le r\right\};
\]
together with the inequality $\max_i \min_j \|X_i-X_j\| \le C_X/n^{1/d}$, its cardinality is lower-bounded by
\begin{align*}
    |\mathcal{P}_{n,h}(x)|
    &\ge 
    \Bigg|
    \left\{
        t=(t_1,t_2,\ldots,t_d) \in \mathbb{Z}^d \setminus \{\mathbf{0}\}
        \, \mid \,
        \left\|\frac{(C_X/n^{1/d})t}{h}\right\| \le r
    \right\}
    \Bigg| \\
    &\ge 
    \Bigg|
    \left\{
        t=(t_1,t_2,\ldots,t_d) \in \mathbb{Z}^d \setminus \{\mathbf{0}\}
        \, \mid \,
        \left\|t\right\| \le \frac{r}{C_X}n^{1/d}h
    \right\}
    \Bigg| \\
    &\ge 
    C'(hn^{1/d})^d \ge C'nh^d
\end{align*}
for some $C'>0$. Therefore, we get
\begin{align*}
    |I_n|^{-1}\sum_{(i_1,i_2) \in I_n} &K_h(x-X_{i_1}) K_h(x'-X_{i_2}) \\
    &= 
    |I_n|^{-1}\sum_{(i_1,i_2) \in I_n \cap \{P_{n,h}(x) \times P_{n,h}(x')\}} 
    K_h(x-X_{i_1}) K_h(x'-X_{i_2}) \\
    &= 
    |I_n|^{-1}\sum_{(i_1,i_2) \in I_n \cap \{P_{n,h}(x) \times P_{n,h}(x')\}} 
    h^{-d}K\left( \frac{x-X_{i_1}}{h} \right) h^{-d}K\left(\frac{x'-X_{i_2}}{h} \right) \\
    &=
    \frac{|I_n \cap \{P_{n,h}(x) \times P_{n,h}(x')\}|}{|I_n|} 
    h^{-2d} (\underline{k} K_{\max})^2 \\
    &\ge 
    \frac{\min\left\{\binom{|\mathcal{P}_{n,h}(x)|}{2},\binom{|\mathcal{P}_{n,h}(x')|}{2}\right\}}{\binom{n}{2}} h^{-2d} (\underline{k}K_{\max})^2 \\
    &\ge 
    C'' \frac{(C'nh^d)^2}{n^2} h^{-2d} (\underline{k}K_{\max})^2 \quad\text{for some }\,C''\\
    &= C'' (C' \underline{k}K_{\max})^2
    =:
    C_K,
\end{align*}
which proves the assertion. 
\end{proof}

The $q$-th moment of $T_{n,h}$ is evaluated as follows. 

\begin{lemma}
\label{lemma: moment of T}
Let $q \in \mathbb{N}$. Provided that $h \le 1$, 
there exist $C_{T,1},C_{T,2} \ge 0$ independent of $n$ and $h$ such that 
\[
    \mathbb{E}(|T_{n,h}-\overline{T}_{n,h}|^q) \le C_{T,1} /(nh^d)^q
    +
    C_{T,2} / (nh^d)^{2q-2}.
\]
\end{lemma}

\begin{proof}
With $\mathbb{E}(K_h(x-X))=\kappa_{1,0}=O(1)$ and 
$k(X,X'):=K_h(x-X)K_h(x'-X')-\kappa_{1,0}^2$ satisfying
\[
    T_{n,h}-\overline{T}_{n,h} = |I_n|^{-1}\sum_{(i_1,i_2) \in I_n} k(X_{i_1},X_{i_2})
    \quad \text{and} \quad 
    \mathbb{E}(k(X,X'))=0,
\]
we have 
\begin{align*}
    \mathbb{E}(|k(X,X')|^q)
    &=
    \mathbb{E}\left(
        \sum_{j=0}^q (-1)^j \binom{q}{j} K_h(x-X)^j K_h(x'-X')^j (\kappa_{1,0}^2)^{q-j}
    \right) \\
    &\le 
    \sum_{j=0}^q
    \binom{q}{j}
    \kappa_{1,0}^{2(q-j)}
    \mathbb{E}(K_h(x-X)^j)
    \mathbb{E}(K_h(x'-X')^j) \\
    &\le 
    \sum_{j=0}^q
    \binom{q}{j}
    \kappa_{1,0}^{2(q-j)}
    \kappa_{j,0} h^{-d(j-1)} \kappa_{j,0} h^{-d(j-1)} \\
    &= 
    \left\{
        1
        + 
        \sum_{j=0}^{q-1} \binom{q}{j} \kappa_{1,0}^{2(q-j)}
        \kappa_{j,0}^2 h^{2d(q-j)}
    \right\}
    h^{-2d(q-1)} \\
&\le 
    C_{T}' h^{-2d(q-1)},
\end{align*}
with $C_{T}':=1+(q-1)\max_j \{\binom{q}{j} \kappa_{1,0}^{2(q-j)} \kappa_{j,0}^2\}$. 
Applying 
Lemma~2.1 of \citet{fu2011exact}  to $k(X,X')$ yields
\begin{align*}
    \mathbb{E}(|T_{n,h}-\overline{T}_{n,h}|^q)
    &\le 
    C_{T,1}' n^{-q} \mathbb{E}(|k(X,X')|^2)^{q/2} 
    +
    C_{T,2}' n^{2-2q} \mathbb{E}(|k(X,X')|^q) \\
    &\le 
    C_{T,1} n^{-q} (h^{-2d(2-1)})^{q/2}
    +
    C_{T,2} n^{2-2q} h^{-2d(q-1)} \\
    &=
    C_{T,1} /(nh^d)^q
    +
    C_{T,2} / (nh^d)^{2q-2}.
\end{align*}
\end{proof}

The $q$-th moment of $\varepsilon_{n,h}$ is also evaluated as follows. 

\begin{lemma}
\label{lemma: moment of epsilon}
Let $q \in \mathbb{N}$. 
Provided that $h \le 1$ and $1/\{\lambda_n(nh^d)^{\nu}\} \le 1$ for some $\nu > 0$, there exist $C_{\varepsilon,1},C_{\varepsilon,2}>0$ independent of $n$ and $h$ such that 
\[
    \mathbb{E}(|\varepsilon_{n,h}|^q) \le C_{\varepsilon,1} / (nh^d)^q + C_{\varepsilon,2}/(nh^d)^{2q-2}.
\]
\end{lemma}

\begin{proof}

$\overline{T}_{n,h} \ge l^2 + \lambda_n$ indicate that
\begin{align}
    \mathbb{E}(|\varepsilon_{n,h}|^q)
    &=
    \mathbb{E}\bigg[
        \frac{
            |T_{n,h}-\overline{T}_{n,h}|^q
        }{
            \overline{T}_{n,h}^q T_{n,h}^q
        }
    \bigg] \nonumber \\
    &=
    \mathbb{E}\bigg[
        \frac{|
            T_{n,h}-\overline{T}_{n,h}
        |^q}{\overline{T}_{n,h}^q T_{n,h}^q}
        \mathbbm{1}\left(
            |T_{n,h}-\overline{T}_{n,h}| \le \frac{\overline{T}_{n,h}}{2}
        \right)
    \bigg] \nonumber \\
    &\hspace{6em}+
    \mathbb{E}\bigg[
        \frac{|
            T_{n,h}-\overline{T}_{n,h}
        |^q}{\overline{T}_{n,h}^q T_{n,h}^q}
        \mathbbm{1}\left(
            |T_{n,h}-\overline{T}_{n,h}| > \frac{\overline{T}_{n,h}}{2}
        \right)
    \bigg] \nonumber \\
&\le 
    \frac{1}{(l^2)^q(l^2/2)^q} \mathbb{E}\bigg[
        |
            T_{n,h}-\overline{T}_{n,h}
        |^q
    \bigg]
    +
    \frac{1}{(l^2)^q \lambda_n^q}
    \mathbb{E}\bigg[
        |
            T_{n,h}-\overline{T}_{n,h}
        |^q
        \mathbbm{1}\left(
            |T_{n,h}-\overline{T}_{n,h}| > \frac{\overline{T}_{n,h}}{2}
        \right)
    \bigg]. \label{eq:epsilon_eq1}
\end{align}
Further, for any $\nu>0$, we have
\begin{align}
\eqref{eq:epsilon_eq1} &\le 
    \frac{2^q}{l^{4q}} \mathbb{E}\bigg[
        |
            T_{n,h}-\overline{T}_{n,h}
        |^q
    \bigg]
    +
    \frac{1}{l^{2q} \lambda_n^q}
    \mathbb{E}\bigg[
        \frac{
            |T_{n,h}-\overline{T}_{n,h}|^{\nu}
        }{
            |T_{n,h}-\overline{T}_{n,h}|^{\nu}
        }
        |T_{n,h}-\overline{T}_{n,h}|^{q}
        \mathbbm{1}\left(
            |T_{n,h}-\overline{T}_{n,h}| > \frac{\overline{T}_{n,h}}{2}
        \right)
    \bigg] \nonumber \\
&\le 
    \frac{2^q}{l^{4q}} \mathbb{E}\bigg[
        \left\{
            T_{n,h}-\overline{T}_{n,h}
        \right\}^q
    \bigg]
    +
    \frac{2^{\nu}}{l^{2q+2\nu} \lambda_n^q}
    \mathbb{E}\bigg[
        (T_{n,h}-\overline{T}_{n,h})^{q+\nu}
    \bigg]. \label{eq:epsilon_eq2}
\end{align}
Therefore, by specifying sufficiently large $\nu=\nu(q)$ satisfying $\lambda_n^q (nh^d)^{\nu} \ge 1$, 
Lemma~\ref{lemma: moment of T} ensures the existence of 
$C_{\varepsilon,1}',C_{\varepsilon,2}',C_{\varepsilon,3}',C_{\varepsilon,4}'>0$ such that 
\begin{align*}
\eqref{eq:epsilon_eq2}
&\le 
    C_{\varepsilon,1}' / (nh^d)^q + C_{\varepsilon,2}'/(nh^d)^{2q-2}
    +
    C_{\varepsilon,3}' / \{\lambda_n^q (nh^d)^{q+\nu}\} + C_{\varepsilon,4}'/\{\lambda_n^q (nh^d)^{2q+2\nu-2}\} \\
&\le 
    C_{\varepsilon,1} / (nh^d)^q + C_{\varepsilon,2}/(nh^d)^{2q-2},
\end{align*}
which proves the assertion. 

\end{proof}

The variance of $S_{n,h}$ is evaluated as follows. 

\begin{lemma}\label{lemma:var_snh}
Provided that $n^{-1/d} \le h \le 1$ and $\lambda_{n} \le h^{d}$, 
there exists $C_{V,S}>0$ independent of $n$ and $h$ such that 
$\Var[S_{n,h}]
\le C_{V,S}h^{2\beta}(nh^{d})^{-1}$.
\end{lemma}

\begin{proof}

To evaluate the variance of $S_{n,h}$, we apply Lemma~\ref{lemma: variance of u} to 
\begin{align*} 
k(X,X')
&=
\{ f(X,X') - f(x,x') \}\{ K_h(x-X)K_h(x'-X') + \lambda_n\} 
\end{align*}
and bound $\zeta_1,\zeta_2$ therein as follows.

\textit{Step A: Bounding $\zeta_1$}. Letting $k^{(1)}(X):=\mathbb{E}[k(X,X') \mid X]$ and
$k^{(2)}(X):=\mathbb{E}[k(X',X) \mid X]$, we have
\begin{align*}
    \zeta_1 
    \le 
    \frac{1}{2}\left( 
        \Var[k^{(1)}(X)]
        +
        \Var[k^{(2)}(X)]
    \right)
    \le 
    \frac{1}{2}\left(
        \mathbb{E} [ \{k^{(1)}(X)\}^2 ]
        +
        \mathbb{E} [ \{k^{(2)}(X)\}^2 ]
    \right).
\end{align*}
By the inequality 
\begin{align*}
|f(X,X')-f(x,x')| &\le |f(X,X')-f(x,X')|+|f(x,X')-f(x,x')| \\
&\le L\norm{x-X}^{\beta}+L\norm{x'-X'}^{\beta},
\end{align*}
we get the following bound on $k^{(1)}(X)$:
\begin{align*}
    |k^{(1)}(X)|
    &\le
    \mathbb{E}\left[ |f(X,X')-f(x,x')|
    \{K_h(x-X)K_h(x'-X') + \lambda_n\} \mid X \right] \\
    &\le 
    L\mathbb{E}\left[ K_h(x-X)K_h(x'-X')\{\norm{x-X}^{\beta}+\norm{x'-X'}^{\beta}\} \mid X\right] \\
    &\hspace{1em}+
    L\lambda_n\mathbb{E}\left[\{\norm{x-X}^{\beta}+\norm{x'-X'}^{\beta}\} \mid X \right].
\end{align*}
By Lemma \ref{lemma: common bound}, 
this is further bounded from above as follows:
\begin{align*}
    |k^{(1)}(X)|
    &\le 
    L \{ \tau_1 + \tau_2 + \tau_3 + \tau_4\},
\end{align*}
where 
\begin{align*}
    \tau_1&:=\kappa_{1,0} K_h(x-X)\norm{x-X}^{\beta}, 
    &\tau_2&:=\kappa_{1,\beta} K_h(x-X) h^{\beta},\\
    \tau_3&:=\lambda_{n} \norm{x-X}^{\beta},
    &\tau_4&:=\lambda_{n} \kappa_{0,\beta}h^{\beta}.
\end{align*}
Then we get
\[ 
\mathbb{E} [\{k^{(1)}(x,X)\}^2] \le L^2 \sum_{j_1=1}^{4}\sum_{j_2=1}^{4} \mathbb{E}[\tau_{j_1} \tau_{j_2}].
\]
Here 
by Lemma~\ref{lemma: common bound},
we bound
the diagonal components $\{\mathbb{E}[\tau_i^2] : i=1,\ldots,4\}$ as
\begin{align*}
    \mathbb{E}[\tau_1^2]
    &= \kappa_{1,0}^2 \mathbb{E}[K_h(x-X)^{2}\norm{x-X}^{2\beta}]
    &\le&
    \kappa_{1,0}^2\kappa_{2,2\beta}\,\, ( h^{-d+2\beta})
    &=&O(h^{-d+2\beta}), \\
    \mathbb{E}[\tau_2^2]
    &= \kappa_{1,\beta}^2 \mathbb{E}[K_h(x-X)^{2}]\,\, h^{2\beta}
    &\le&
    \kappa_{1,\beta}^2 \kappa_{2,0}\,\, (h^{-d+2\beta})
    &=&O(h^{-d+2\beta}), \\
    \mathbb{E}[\tau_3^2]
    &= \lambda_{n}^{2} \,\, \mathbb{E}[K_h(x-X)^{0}\norm{x-X}^{2\beta}]
    &\le&
     \kappa_{0,2\beta} \,\, (\lambda_n^2 h^{d+2\beta})
    &=&O(\lambda_n^2 h^{d+2\beta}), \\
    \mathbb{E}[\tau_4^2]
    &= \lambda_{n}^{2} \,\, \kappa_{0,\beta}^2 \,\, h^{2\beta}
    &\le&
    \kappa_{0,\beta}^2 \,\,  (\lambda_n^2 h^{d+2\beta})
    &=&O(\lambda_n^2 h^{d+2\beta}).
\end{align*}
Similarly,
we bound the off-diagonal components $\{\mathbb{E}[\tau_i \tau_j] : 1\le i<j\le 4\}$ as
\begin{align*}
    \mathbb{E}[\tau_1\tau_2]
    &=
    \kappa_{1,0}\kappa_{1,\beta} \mathbb{E}[K_h(x-X)^{2}\norm{x-X}^{\beta}]\,\, h^{\beta}
    &\le&\kappa_{1,0} \kappa_{1,\beta}\kappa_{2,\beta} \,\, (h^{-d+2\beta})
    &=&O(h^{-d+2\beta}), \\
    \mathbb{E}[\tau_1 \tau_3]
    &=
    \lambda_{n} \,\, \kappa_{1,0} \mathbb{E}[K_h(x-X)\norm{x-X}^{2\beta}]
    &\le&
     \kappa_{1,0}\kappa_{1,2\beta} \,\, (\lambda_n h^{2\beta})
    &=&O(\lambda_n h^{2\beta}), \\
    \mathbb{E}[\tau_1 \tau_4]
    &=
     \lambda_{n}\,\, \kappa_{1,0}\kappa_{0,\beta} \mathbb{E}[K_h(x-X)\norm{x-X}^{\beta}] \,\,h^{\beta}
    &\le&
    \kappa_{1,0}\kappa_{0,\beta}\kappa_{1,\beta} \,\, (\lambda_n  h^{2\beta})
    &=&O(\lambda_n h^{2\beta}), \\
    \mathbb{E}[\tau_2 \tau_3]
    &=
    \lambda_{n} \,\, \kappa_{1,\beta} \mathbb{E}[K_h(x-X)\norm{x-X}^{\beta}] \,\, h^{\beta} 
    &\le&
    \kappa_{1,\beta}\kappa_{1,\beta} \,\, (\lambda_n  h^{2\beta})
    &=&O(\lambda_n h^{2\beta}), \\
    \mathbb{E}[\tau_2 \tau_4]
    &=
    \lambda_{n}\,\, \kappa_{1,\beta}\kappa_{0,\beta} \mathbb{E}[K_h(x-X)] \,\, h^{2\beta}
    &\le&
     \kappa_{1,\beta}\kappa_{0,\beta}\kappa_{1,0} \,\, (\lambda_n h^{2\beta})
    &=&O(\lambda_n h^{2\beta}), \\
    \mathbb{E}[\tau_3 \tau_4]
    &=
    \lambda_{n}^{2}\,\, \kappa_{0,\beta} \mathbb{E}[K_h(x-X)^{0}\norm{x-X}^{\beta}] \,\, h^{\beta} 
    &\le&
    \kappa_{0,\beta} \kappa_{0,\beta} \,\, (\lambda_n^2 h^{d+2\beta})
    &=&O(\lambda_n^2 h^{d+2\beta}).
\end{align*}
These inequalities indicate that there exist positive constants $D_{1},D_{2},D_{3}$ depending only on $\{\kappa_{i,j\beta}: i=0,1, j=0,1,2\}$ for which we have
\begin{align*}
    \mathbb{E}[\{k^{(1)}(x,X)\}^2]
    &\le
    L^2 \{
        D_1 h^{-d+2\beta}
        +
        D_2 \lambda_n h^{2\beta}
        +
        D_3 \lambda_n^2 h^{d+2\beta}
    \}\\
    &\le
    L^2 h^{-d+2\beta}\{D_1+ D_2 \lambda_n h^d + D_3 \lambda_n^2 h^{2d}\} \\
    &\le 
    D h^{-d+2\beta} \text{  with $D=L^{2}(D_{1}+D_{2}+D_{3})$},
\end{align*}
where we use $\lambda_{n}\le h^{d}$ and $h\le1$. 
Similarly, 
there exists a positive constant $D'$ depending only on $L$ and $\{\kappa_{i,j\beta}: i=0,1, j=0,1,2\}$ for which we have
\[
\mathbb{E}[\{k^{(2)}(x,X)\}^2] \le D' h^{-d+2\beta}.
\]
Consequently, we obtain 
\begin{align}
    \zeta_1
    &\le 
    \frac{1}{2}\left\{
        \mathbb{E}[\{k^{(1)}(x,X)\}^2]
        +
        \mathbb{E}[\{k^{(2)}(x,X)\}^2]
    \right\} 
    =
    \frac{D+D'}{2} \,\, h^{-d+2\beta}.
    \label{eq: bound on zeta1}
\end{align}

\vspace{5mm}

\textit{Step B: Bounding $\zeta_2$}.
By the H\"{o}lder continuity of $f$, 
we have
\begin{align*}
    \zeta_2
    \le
    \Var[k(X,X')]
    &\le 
    \mathbb{E}[k(X,X')^2] \\
    &\le 
    \mathbb{E}\left[
        \{K_h(x-X) K_h(x'-X')+ \lambda_n \}^2
        |f(X,X')-f(x,x')|^2
    \right] \\
    &\le 
    \mathbb{E}\left[
        \{K_h(x-X) K_h(x'-X') + \lambda_n \}^2
        \{L\norm{x-X}^{\beta}+L\norm{x'-X'}^{\beta}\}^2
    \right] \\
    &\le 
    L^2\sum_{j_1=0}^2 \sum_{j_2=0}^2
    \binom{2}{j_1}\binom{2}{j_2} \xi_{j_1,j_2},
\end{align*}
where 
\begin{align*}
\xi_{k,l}
&:=
\mathbb{E}[\{K_h(x-X) K_h(x'-X')\}^{j_1} \lambda_n^{2-j_1} \norm{x-X}^{j_2\beta}\norm{x'-X'}^{(2-j_2)\beta}].
\end{align*}
Here Lemma \ref{lemma: common bound} gives
\begin{align*}
\xi_{k,l}
&=
\lambda_n^{2-j_1}
\mathbb{E}[K_h(x-X)^{j_1} \norm{x-X}^{j_2\beta}]
\mathbb{E}[K_h(x'-X')^{j_1} \norm{x'-X'}^{(2-j_2)\beta}] \\
&\le 
\lambda_n^{2-j_1} \,\,
\kappa_{j_1,j_2\beta}
\kappa_{j_1,(2-j_2)\beta}\,\, h^{(1-j_1)d+ j_{2}\beta} h^{(1-j_1)d+(2-j_{2})\beta} \\
&=
\kappa_{j_1,j_2\beta}\kappa_{j_1,(2-j_2)\beta}
\lambda_n^{2-j_1}
h^{2(1-j_1)d+2\beta} \\
&\le 
D_{4} \,\,h^{-2d+2\beta},
\end{align*}
where $D_{4}$ is a positive constant depending only on $\{\kappa_{j_1,j_2\beta}: j_{1}=0,1,2, j_{2} =0,1,2\}$.
This concludes that there exists a positive constant depending only on 
$L$ and 
$\{\kappa_{j_1,j_2\beta}: j_{1}=0,1,2, j_{2}$
for which we have
\begin{align}
\zeta_2\le D'' h^{-2d+2\beta}.
\label{eq: bound on zeta2}
\end{align}

\textit{Final step: Combining bounds on $\zeta_{1}$ and $\zeta_{2}$}.
Combining (\ref{eq: bound on zeta1}) and (\ref{eq: bound on zeta2}) yields 
\begin{align*}
    \Var[S_{n,h}]
    &=
    \frac{4(n-2)}{n(n-1)}\zeta_1 + \frac{2}{n(n-1)}\zeta_2
    &\le&
    2(D+D')
    h^{2\beta} (nh^{d})^{-1}
    \,\, + \,\, 4D'' h^{2\beta}(nh^{d})^{-2}\\
    & 
    &\le& (2D+2D'+4D'')
        h^{2\beta} (nh^d)^{-1}
        (1+(nh^d)^{-1})
\end{align*}
which completes the proof.
\end{proof}

We last evaluate the fourth moment of $S_{n,h}$ as follows. 

\begin{lemma}\label{lemma:var_snhenh} 
Provided that $n^{-1/d} \le h \le 1$, 
there exists $C_{E,S}>0$ independent of $n$ and $h$ such that $\mathbb{E}(S_{n,h}^4) \le C_{E,S}h^{4\beta}$.
\end{lemma}

\begin{proof}[Proof of Lemma \ref{lemma:var_snhenh}]

To evaluate $b_n:=\mathbb{E}(S_{n,h}^4)$, we prepare the multi-index notation: let
\[
    \mathcal{J}_k(m)
    :=
    \{(j_1,j_2,\ldots,j_k) \in \{0,1,2,\ldots,m\} \mid j_1+j_2+\cdots+j_k=m\}, 
\]
and let
\[
\binom{m}{j_1,j_2,\ldots,j_k}=\frac{m!}{j_1!j_2!\cdots j_k!}.
\]

Let
\[
h_{i_{1},i_{2}}(X,X')
:=
\{\norm{x-X}^{\beta}+\norm{x'-X'}^{\beta}\}\{K_{h}(x-X)K_{h}(x'-X') + \lambda_n\}
\]
and 
let
$X^{(1)}$ and $X^{(2)}$ be independent copies of $X$.
By the inequality $|f(X,X')-f(x,x')| \le L\norm{x-X}^{\beta}+L\norm{x'-X'}^{\beta}$
and 
by Lemma \ref{lemma: Rosenthal},
we get
\begin{align*}
    b_{n}
    &=
    \frac{1}{|I_n|^{4}}\mathbb{E}\Bigg[ \Bigg(
    \sum_{(i_{1},i_{2})\in I_{n}}
    |f(X_{i_{1}},X_{i_{2}})-f(x,x')|\{K_{h}(x-X_{i_{1}})K_{h}(x'-X_{i_{2}}) + \lambda_n \} \Bigg)^4 \Bigg] \\
    & \le \frac{L^{4}}{|I_{n}|^{4}} \mathbb{E} \Bigg[ \Bigg( \sum_{(i_{1},i_{2})\in I_{n}} 
    h_{i_{1},i_{2}}(X_{i_{1}},X_{i_{2}}) \Bigg)^{4} \Bigg]
    \\
    & \le \frac{C_{R}L^{4}}{|I_{n}|^{4}}
    \Bigg\{
    \underbrace{
    \Bigg(
        \sum_{(i_1,i_2)\in I_{n}}\mathbb{E}\left[
            h_{i_1,i_2}(X^{(1)}_{i_{1}},X^{(2)}_{i_{2}})
        \right] 
    \Bigg)^4}_{=:E_{1}}
    \\
    &\hspace{4em}
    +
    \underbrace{\sum_{1\le i_1\le n}\mathbb{E}\Bigg[ \Bigg( \sum_{1\le i_2\le n}\mathbb{E}[ h_{i_1,i_2}(X^{(1)}_{i_{1}},X^{(2)}_{i_{2}}) \mid X^{(1)}_{i_1}] \Bigg)^4 \Bigg]}_{=:E_{2}}
    \\
    & \hspace{4em}
    +
    \underbrace{\sum_{1\le i_2 \le n}\mathbb{E}\Bigg[ \Bigg( \sum_{1\le i_1\le n}\mathbb{E}[ h_{i_1,i_2}(X^{(1)}_{i_{1}},X^{(2)}_{i_{2}}) \mid X^{(2)}_{i_2}] \Bigg)^4 \Bigg]}_{=:E_{3}}
    +
    \underbrace{\sum_{(i_1,i_2)\in I_{n}}\mathbb{E}\Bigg[
        h_{i_1,i_2}^4 (X^{(1)}_{i_{1}},X^{(2)}_{i_{2}}) \Bigg]}_{=:E_{4}}
        \Bigg\}.
\end{align*}
We bound $E_{1}$, by using the identity
\begin{align*}
\mathbb{E}\big[h_{i_{1},i_{2}}(X^{(1)}_{i_{1}},X^{(2)}_{i_{2}})\big]
&=
\mathbb{E}\big[ K_h(x-X_{i_1}^{(1)})\norm{x-X_{i_1}^{(1)}}^{\beta} \big]\,
\mathbb{E}\big[ K_h(x'-X_{i_2}^{(2)})\norm{x'-X_{i_2}^{(2)}}^{0}\big]\\
&\quad +
        \mathbb{E}\big[ K_h(x-X_{i_1}^{(1)}) \norm{x-X_{i_1}^{(1)}}^{0} \big]\,
        \mathbb{E}\big[ K_h(x'-X_{i_2}^{(2)})\norm{x'-X_{i_2}^{(2)}}^{\beta} \big] \\
&\quad +\lambda_n \mathbb{E}\big[K_{h}(x-X_{i_{1}}^{(1)})^{0}\norm{x-X_{i_1}^{(1)}}^{\beta}\big]
        +\lambda_n \mathbb{E}\big[K_{h}(x'-X_{i_{2}}^{(2)})^{0}\norm{x'-X_{i_2}^{(2)}}^{\beta}\big]
\end{align*}
and using Lemma \ref{lemma: common bound},
as
\begin{align*}
    E_{1}^{1/4}
    &\le 
    \sum_{(i_1,i_2)\in I_{n}}\bigg(
        \kappa_{1,\beta}
        \kappa_{1,0} \,\, h^{\beta}
        +
        \kappa_{1,0}
        \kappa_{1,\beta} \,\, h^{\beta}
        +
        \kappa_{0,\beta} \,\, \lambda_n h^{d+\beta}
        +
        \kappa_{0,\beta}\,\, \lambda_n h^{d+\beta}
    \bigg) \nonumber\\
    &=
    \sum_{(i_1,i_2)\in I_{n}}\bigg(
        2\kappa_{1,\beta}\kappa_{1,0} \,\, h^{\beta}
        +
        2 \kappa_{0,\beta} \,\, \lambda_n h^{d+\beta}
    \bigg) \nonumber\\
    &= 
    n(n-1)
    \bigg(
        2\kappa_{1,\beta}\kappa_{1,0}h^{\beta}
        +
        2\lambda_n \kappa_{0,\beta}h^{d+\beta}
    \bigg) \nonumber\\
    &= F_{1} \,\, n^{2} h^{\beta}
    \text{ with $F_{1}:=2(\kappa_{1,\beta}\kappa_{1,0}+\kappa_{0,\beta})$},
\end{align*}
which implies 
\begin{align}
    E_{1} \le F_{1}^{4} \,\, n^{8}h^{4\beta}.
    \label{eq: bound on E1}
\end{align}
We bound $E_{2}$, by using the inequality
\begin{align*}
    \mathbb{E}&[ h_{i_1,i_2}(X^{(1)}_{i_{1}},X^{(2)}_{i_{2}}) \mid X^{(1)}_{i_1}]
    \nonumber\\
    &=
    \norm{x-X_{i_1}^{(1)}}^{\beta}K_h(x-X_{i_1}^{(1)})
        \mathbb{E}\big[K_h(x'-X_{i_2}^{(2)})\big]
    &+&
    K_h(x-X_{i_1}^{(1)})
        \mathbb{E}\big[K_h(x'-X_{i_2}^{(2)}) \norm{x'-X_{i_2}^{(2)}}^{\beta}\big]
    \nonumber\\
    &\hspace{1em}+
    \lambda_n \norm{x-X_{i_1}^{(1)}}^{\beta}
    &+&
    \lambda_n \mathbb{E}\big[K_h(x'-X_{i_2}^{(2)})^{0}\norm{x'-X_{i_2}^{(2)}}^{\beta}\big]
    \nonumber\\
    &\le 
    \kappa_{1,0} \,\, \norm{x-X_{i_1}^{(1)}}^{\beta}K_h(x-X_{i_1}^{(1)})
    &+&
    \kappa_{1,\beta}\,\, h^{\beta}\,\, K_h(x-X_{i_1}^{(1)})
    \nonumber\\
    &\hspace{1em}+
    \lambda_n \norm{x-X_{i_1}^{(1)}}^{\beta}
    &+&
    \lambda_n \,\, \kappa_{0,\beta} \,\, h^{d+\beta},
\end{align*}
as
\begin{align*}
    E_{2}
    &\le 
    n^{4}
    \sum_{1\le i_1 \le n}\mathbb{E}
    \Bigg[
    \Bigg(
    \kappa_{1,0} \,\, \norm{x-X_{i_1}^{(1)}}^{\beta}K_h(x-X_{i_1}^{(1)})
    +\kappa_{1,\beta}\,\, h^{\beta}\,\, K_h(x-X_{i_1}^{(1)})
    \nonumber\\
    &\hspace{16em}
    +\lambda_n \norm{x-X_{i_1}^{(1)}}^{\beta}
    +\lambda_n \,\, \kappa_{0,\beta} \,\, h^{d+\beta}
    \big]
    \Bigg)^{4}
    \Bigg]
    \nonumber\\
    &=
    n^{4} \sum_{1 \le i_1 \le n}
    \sum_{(j_1,j_2,j_3,j_4) \in \mathcal{J}_4(4)}
    \binom{4}{j_1,j_2,j_3,j_4} \mathcal{K}_{j_{1},j_{2},j_{3},j_{4}},
\end{align*}
where
\begin{align*}
    \mathcal{K}_{j_{1},j_{2},j_{3},j_{4}}
    &:=
    \mathbb{E}
    \bigg[
        \bigg\{\norm{x-X_{i_1}^{(1)}}^{\beta}K_h(x-X_{i_1}^{(1)})
        \kappa_{1,0}\bigg\}^{j_1}
        \bigg\{K_h(x-X_{i_1}^{(1)}) \kappa_{1,\beta}h^{\beta}\bigg\}^{j_2}
    \\
    &\hspace{12em}
        \cdot \bigg\{\lambda_n \norm{x-X_{i_1}^{(1)}}^{\beta}\bigg\}^{j_3}
        \bigg\{\lambda_n \kappa_{0,\beta}h^{d+\beta}\bigg\}^{j_4}
    \bigg].
\end{align*}
Since we have
\begin{align*}
\mathcal{K}_{j_{1},j_{2},j_{3},j_{4}}
&=
\kappa_{1,0}^{j_{1}}
\kappa_{1,\beta}^{j_{2}}
\kappa_{0,\beta}^{j_{4}}\,\,
\mathbb{E}\bigg[
K_{h}(x-X_{i_{1}}^{(1)})^{j_{1} + j_{2} }
|x-X_{i_{1}}^{(1)}|^{ (j_{1}+j_{3}) \beta }
\bigg]
\,\,
\lambda_{n}^{j_{3}+j_{4}}
h^{ j_{2} \beta + j_{4} (d+\beta) }
\\
&\le
\kappa_{1,0}^{j_{1}}
\kappa_{1,\beta}^{j_{2}}
\kappa_{0,\beta}^{j_{4}} \kappa_{ j_{1}+j_{2}, (j_{1}+ j_{3})\beta }
\,\,
\lambda_{n}^{j_{3}+j_{4}}
h^{(1-j_{1}-j_{2})d + (j_{1}+j_{3})\beta + j_{2}\beta + j_{4}(d+\beta)}
\\
&=
\kappa_{1,0}^{j_{1}}
\kappa_{1,\beta}^{j_{2}}
\kappa_{0,\beta}^{j_{4}} \kappa_{ j_{1}+j_{2}, (j_{1}+ j_{3})\beta }
\,\,
\lambda_{n}^{j_{3}+j_{4}}
h^{(1-j_{1}-j_{2}+j_{4})d + (j_{1}+j_{2}+j_{3}+j_{4} )\beta }
,
\end{align*}
we further bound $E_{2}$ as
\begin{align}
    E_{2}
    &\le
    n^{5}
    \sum_{(j_1,j_2,j_3,j_4) \in \mathcal{J}_4(4)}
    \binom{4}{j_1,j_2,j_3,j_4}
    \kappa_{1,0}^{j_{1}}
    \kappa_{1,\beta}^{j_{2}}
    \kappa_{0,\beta}^{j_{4}} \kappa_{ j_{1}+j_{2}, (j_{1}+ j_{3})\beta }
    \,\,
    \lambda_{n}^{j_{3}+j_{4}}
    h^{(1-j_{1}-j_{2}+j_{4})d + (j_{1}+j_{2}+j_{3}+j_{4} )\beta }
    \nonumber\\
    &=
    n^{5}
    \sum_{(j_1,j_2,j_3,j_4) \in \mathcal{J}_4(4)}
    \omega_{j_{1},j_{2},j_{3},j_{4}}
    \lambda_{n}^{j_{3}+j_{4}}
    h^{(j_{3}+2j_{4}-3)d + 4\beta }
    \nonumber\\
    &\le
    F_{2} \,\,n^5 h^{-3d+4\beta},
    \label{eq: bound on E2}
\end{align}
where we let
\begin{align*}
    \omega_{j_{1},j_{2},j_{3},j_{4}}
    &:=
    \binom{4}{j_1,j_2,j_3,j_4}
    \kappa_{1,0}^{j_{1}}
    \kappa_{1,\beta}^{j_{2}}
    \kappa_{0,\beta}^{j_{4}} \kappa_{ j_{1}+j_{2}, (j_{1}+ j_{3})\beta },
    \\
    F_{2}
    &:= 4^{4}\max_{(j_{1},j_{2},j_{3},j_{4})\in \mathcal{J}(4)}
    \kappa_{1,0}^{j_{1}}
    \kappa_{1,\beta}^{j_{2}}
    \kappa_{0,\beta}^{j_{4}} \kappa_{ j_{1}+j_{2}, (j_{1}+ j_{3})\beta }.
\end{align*}
Similarly, we get
\begin{align}
    E_{3} \le F_{2}\,\, n^{5} h^{-3d+4\beta}
    \label{eq: bound on E3}
\end{align}
as well as
\begin{align}
    E_{4}
    &=
     \sum_{(i_1,i_2)\in I_{n}}
        \sum_{(j_1,j_2) \in \mathcal{J}_2(4)}
        \sum_{(j_1',j_2') \in \mathcal{J}_2(4)}
        \binom{4}{j_1,j_2}
        \binom{4}{j_1',j_2'} \nonumber\\
    &\hspace{3em}
    \mathbb{E}\bigg[
        \norm{x-X_{i_1}}^{j_1\beta}
        K_h(x-X_{i_1})^{j_1'}
    \bigg]
    \mathbb{E}\bigg[
        \norm{x'-X_{i_2}}^{j_2\beta}
        K_h(x'-X_{i_2})^{j_1'}
    \bigg]
    \lambda_n^{j_2'} \nonumber\\
    &=
     \sum_{(i_1,i_2)\in I_{n}}
        \sum_{(j_1,j_2) \in \mathcal{J}_2(4)}
        \sum_{(j_1',j_2') \in \mathcal{J}_2(4)}
        \binom{4}{j_1,j_2}
        \binom{4}{j_1',j_2'}
        \kappa_{j_1',j_1\beta}h^{(1-j_1')d+j_1\beta}
        \kappa_{j_1'j_2\beta}h^{(1-j_1')d+j_2\beta}
        \lambda_n^{j_2'} \nonumber\\
    &\le
    F_{3}\,\,
        n^2 
        \sup_{j_2' \in \{0,1,2,3,4\}}
        \lambda_n^{j_2'}
        h^{2(-3+j_2')d+4\beta}
     \quad (\text{ for some $F_{3}>0$ not depending on $n$ and $h$}) \nonumber\\
    &\le
    F_{3} \,\, n^{2} h^{-6d+4\beta}.
    \label{eq: bound on E4}
\end{align}

Combining (\ref{eq: bound on E1})--(\ref{eq: bound on E4}), we obtain
\begin{align}
    b_n
    &\le \frac{2^{4}}{n^{8}} 
    (F_{1}^{4} n^{8}h^{4\beta} + 2F_{2} n^{5}h^{-3d+4\beta} + F_{3} n^{2}h^{-6d+4\beta} )
    \nonumber\\
    &\le G_{1} \,\, h^{4\beta} \left(1 + \frac{1}{(nh^{d})^{3}} + \frac{1}{(nh^{d})^{6}} \right) 
    \text{ with } G_{1}:= 2^{4}F_{1}^{4}+2^{5}F_{2} + F_{3},
    \label{eq: bound on b^1}
\end{align}
which concludes the proof.

\end{proof}

\subsection{Proof of Theorem~\ref{proposition: upper bounds of terms}}
\label{app:proof of proposition: upper bounds of terms}

Using the supporting lemmas in the previous subsection, we shall give upper estimates of  $\rho_{1}$, $\rho_{2}$, and $\rho_{3}$.

\subsubsection*{Step 1: Bounding $\rho_1(n,h)$}
We start with employing the Jensen inequality to get
\[ 
| f(x,x')-\mathbb{E}[\hat{f}_{n,h}(x,x')] |^{2}
\le \mathbb{E}\left[ \left| f(x,x')-\mathbb{E}\left[\hat{f}_{n,h}(x,x')\mid X\right] \right|^{2} \right]
.
\]

Under the condition~\ref{condition: design} (a): with the diameter $\text{diam}(\mathcal{X}):=\sup\{\|x-x'\| \mid x,x' \in \mathcal{X}\}<\infty$ of the set $\mathcal{X}$, expression~(\ref{eq: sum expression of f hat}) gives
\begin{align*}
    \Big|f(x,x')-&\mathbb{E}\Big[\hat{f}_{n,h}(x,x') \mid X\Big]\Big| \\
    &\le 
    \Big| \sum_{(i_1,i_2) \in I_n} W_{i_1,i_2}(x,x';X)\{f(x,x')-f(X_{i_1},X_{i_2})\}\Big| \\
    &\le 
    \sum_{(i_1,i_2) \in I_n} W_{i_1,i_2}(x,x';X)|f(x,x')-f(X_{i_1},X_{i_2})\}| \\
    &\le 
    \sum_{(i_1,i_2) \in I_n} W_{i_1,i_2}(x,x';X)|L\|x-X_{i_1}\|^{\beta} + L\|x'-X_{i_2}\|^{\beta}\}| \\
    &\le 
    \sum_{(i_1,i_2) \in I_n} \frac{K_h(x-X_{i_1})K_h(x'-X_{i_2})}{\sum_{(i_1,i_2) \in I_n}\{K_h(x-X_{i_1})K_h(x'-X_{i_2}) + \lambda_n\} }|L\|x-X_{i_1}\|^{\beta} + L\|x'-X_{i_2}\|^{\beta}| \\
    &\hspace{2em}+
     \sum_{(i_1,i_2) \in I_n} \frac{\lambda_n|L\|x-X_{i_1}\|^{\beta} + L\|x'-X_{i_2}\|^{\beta}|}{\sum_{(i_1,i_2) \in I_n}\{K_h(x-X_{i_1})K_h(x'-X_{i_2}) + \lambda_n\}} \\
    &\le 
    2Lh^{\beta}
    \frac{\sum_{(i_1,i_2) \in I_n} K_h(x-X_{i_1})K_h(x'-X_{i_2})}{\sum_{(i_1,i_2) \in I_n}\{K_h(x-X_{i_1})K_h(x'-X_{i_2}) + \lambda_n\}}\\
    &\hspace{2em}+
    2L\text{diam}(\mathcal{X})^{\beta} 
    \frac{\lambda_n}{|I_n^{-1}|\sum_{(i_1,i_2) \in I_n}K_h(x-X_{i_1})K_h(x'-X_{i_2}) + \lambda_n} \\
    &\le 
    2Lh^{\beta} + 2L\text{diam}(\mathcal{X})^{\beta} \frac{\lambda_n}{c+\lambda_n} 
    \qquad (\because \text{Lemma~\ref{lemma: number of positive points}})  \\
    &\le 
    C_{1} h^{\beta}
\end{align*}
for some $C_{1}>0$, where the last inequality follows from the assumptions $h \le 1$ and $\lambda_n \le h^d$, indicating $\lambda_n \le h^{d} \le h \le h^{\beta}$ (for $0 < \beta \le 1 \le d$).  

Under the condition~\ref{condition: design} (b): 
expression~(\ref{eq: frac expression}) gives
\begin{align*}
    \Big|f(x,x')-&\mathbb{E}\Big[\hat{f}_{n,h}(x,x') \mid X\Big]\Big| 
=
    \Big|
    \frac{S_{n,h}}{T_{n,h}}
    \Big| 
=
    \Big|
    \frac{S_{n,h}}{\overline{T}_{n,h}}
    \Big|
    +
    \Big|
    S_{n,h} \varepsilon_{n,h}
    \Big| 
\le 
    \frac{|S_{n,h}|}{l^2}
    +
    \Big| S_{n,h} \varepsilon_{n,h} \Big|. 
\end{align*}
This, together with Lemma~\ref{lemma: moment of epsilon}, \ref{lemma:var_snh}, \ref{lemma:var_snhenh}, indicates that
\begin{align*}
    \mathbb{E}\left[ \left| f(x,x')-\mathbb{E}\left[\hat{f}_{n,h}(x,x')\mid X\right] \right|^{2} \right]
&\le 
    \mathbb{E}\left(
        \left\{
        \frac{|S_{n,h}|}{l^2}
        +
        \Big| S_{n,h} \varepsilon_{n,h} \Big|
        \right\}^2
    \right) \\
&\le 
    2  
    \left\{
        \frac{\mathbb{E}(S_{n,h}^2)}{l^4}
        +
        \mathbb{E}(S_{n,h}^2 \varepsilon_{n,h}^2)
    \right\} \\
&\le 
    2 \left\{
        \frac{\Var[S_{n,h}] + \mathbb{E}[|S_{n,h}|]^2}{l^4}
        +
        \mathbb{E}[S_{n,h}^4]^{1/2}\mathbb{E}[\varepsilon_{n,h}^4]^{1/2}
    \right\} \\
&\overset{(\star)}{\le} 
    2\left\{
        \frac{C_2 h^{4\beta} + C_3 h^{2\beta} (1+\lambda_n h^{-d})^2}{l^2}
        +
        C_4 (h^{4\beta})^{1/2}(1/(nh^d)^3)^{1/2}
    \right\} \\
&\le 
    C_{5} h^{2\beta}
\end{align*}
for some $C_{2},C_3,C_4,C_5>0$, where in the inequality $(\star)$ we further utilize the evaluation
\begin{align*}
    \mathbb{E}(|S_{n,h}|)
&=
    \mathbb{E}\left(
    \bigg|
        |I_n|^{-1}\sum_{(i_1,i_2) \in I_n} 
        \{f(x,x')-f(X_{i_1},X_{i_2})\}
        \{K_h(x-X_{i_1})K_h(x'-X_{i_2})+\lambda_n\}
    \bigg|
    \right) \\
&\le 
    \mathbb{E}\left(
        |I_n|^{-1}\sum_{(i_1,i_2) \in I_n} 
        \big| f(x,x')-f(X_{i_1},X_{i_2})\big|
        \{K_h(x-X_{i_1})K_h(x'-X_{i_2})+\lambda_n\}
    \right) \\
&\le 
    |I_n|^{-1}\sum_{(i_1,i_2) \in I_n} 
    \mathbb{E}\left(
        \{ L \|x-X_{i_1}\|^{\beta} + L \|x'-X_{i_2}\|^{\beta} \}
        \{K_h(x-X_{i_1})K_h(x'-X_{i_2})+\lambda_n\}
    \right) \\
&=
    \mathbb{E}\left(
        \{ L \|x-X\|^{\beta} + L \|x'-X'\|^{\beta} \}
        \{K_h(x-X)K_h(x'-X')+\lambda_n\}
    \right) \\
&\le 
    2 L \kappa_{1,0}\kappa_{1,\beta} h^{\beta} 
    +
    2 \lambda_n L \kappa_{0,\beta} h^{-d+\beta} \quad (\because \text{Lemma \ref{lemma: common bound}}).
\end{align*}

Therefore, we obtain
\begin{align*}
    \rho_1(n,h)
    =
    \sup_{f \in \mathcal{F}(\beta,L)}
    \left|f(x,x')-\mathbb{E}\left[\hat{f}_{n,h}(x,x')\right]\right|^2
    \le 
    C_{5} h^{2\beta}.
\end{align*}

\bigskip 

\subsubsection*{Step 2: Bounding $\rho_2(n,h)$
under Condition \ref{condition: design} (a).}
Observe that, for any $x,x',i_1,i_2$,
\begin{align}
W_{i_1,i_2}(x,x';X) 
&\le \frac{K_{h}(x-X_{i_{1}})K_{h}(x'-X_{i_{2}}) + \lambda_n}{ 
\sum_{ j_{1},j_{2} : \|x-X_{j_{1}}\| \le rh ,
\|x'-X_{j_{2}}\| \le rh
}
\{K_{h}(x-X_{j_{1}})K_{h}(x'-X_{j_{2}})+\lambda_n\} }
\nonumber\\
&\le 
\frac{  h^{-2d}K_{\mathrm{max}}^{2} +\lambda_{n} }
{ \underk^{2} h^{-2d}K_{\mathrm{max}}^{2} +\lambda_{n} }
\frac{1}{|\{j_{1} :  \|x-X_{j_{1}}\| \le rh \}| \cdot 
|\{j_{2} :  \|x'-X_{j_{2}}\| \le rh \}|
}
\nonumber\\
&\le 
\max\left\{\frac{1}{\underk^{2}}, 1 \right\}
\frac{1}{|\{j_{1} :  \|x-X_{j_{1}}\| \le rh \}| \cdot 
|\{j_{2} :  \|x'-X_{j_{2}}\| \le rh \}|
}
\label{eq: bound W}
\end{align}
where the second inequality follows from Condition \ref{condition: kernel}
and the last inequality follows from the inequality
$(a+b)/(c+d)\le \max\{a/c,b/d\}$ for $a,b,c,d>0$.
Combined with Condition \ref{condition: design} (a),
this gives
\begin{align*}
\max_{i_{1},i_{2}}
W_{i_1,i_2}(x,x';X) 
&\le 
\max\left\{\frac{1}{\underk^{2}}, 1 \right\}
\frac{1}{r^{2d} n^{2}h^{2d}}.
\end{align*}
Combined with Condition \ref{condition: link regression function} (b) and (\ref{eq: sum of W}), this yields
\begin{align*}
    \Var[\hat{f}_{n,h}(x,x') \mid X]
    &=
    \Var\left[
    \sum_{(i_1,i_2) \in I}Y_{i_1i_2}W_{i_1,i_2}(x,x';X)
    \big| X \right] \\
    &=
    \sum_{(i_1,i_2) \in I}W_{i_1,i_2}^2(x,x';X)\sigma^2(X_{i_1},X_{i_2}) \\
    &\le 
    \tau \left\{\sum_{(i_1,i_2) \in I}W_{i_1,i_2}(x,x';X)\right\}
    \max_{i_1,i_2}W_{i_1,i_2}(x,x';X) \\
    &\le 
    C_6 \frac{1}{n^2h^{2d}} \text{ with } C_6:=
    \max\left\{\frac{1}{\underk^{2}}, 1 \right\}
    \frac{\tau}{r^{2d}},
\end{align*}
which gives the desired bound under Condition \ref{condition: design} (a):
\begin{align*}
    \rho_2(n,h)
    &=
    \sup_{f \in \mathcal{F}(\beta,L)}
    \mathbb{E}[\Var[\hat{f}_{n,h}(x,x') \mid X]] 
    \le 
    C_6 (nh^{d})^{-2}. 
\end{align*}

\subsubsection*{Step 2': Bounding $\rho_2(n,h)$
under Condition \ref{condition: design} (b).}
We start with the inequality
\begin{align*}
    \mathbb{E}[\Var[\hat{f}_{n,h}(x,x') \mid X]]
    &\le 
    \tau \mathbb{E}\left[\left\{\sum_{(i_1,i_2) \in I}W_{i_1,i_2}(x,x';X)\right\}
    \max_{i_1,i_2}W_{i_1,i_2}(x,x';X)\right]\\
    &=\tau\mathbb{E}\left[
    \max_{i_1,i_2}W_{i_1,i_2}(x,x';X)\right].
\end{align*}
From (\ref{eq: bound W}),
this is further bounded as
\begin{align*}
    \mathbb{E}[\Var[\hat{f}_{n,h}(x,x') \mid X]]
    &\le 
    \tau \max\left\{\frac{1}{\underk^{2}}, 1 \right\} 
    \mathbb{E}\left[\min\left\{1, \frac{1}{B(n,v(x))B(n,v(x'))}\right\}\right],
\end{align*}
where
let $B(n,v(x)):=
|\{j  : 1\le j \le n,  \|x-X_{j}\| \le rh \}|$
and $B(n,v(x')):=
|\{j : 1\le j \le n, \|x'-X_{j}\| \le rh \}|$.
Note that 
$B(n,v(x))$ and $B(n,v(x'))$ are dependent and binomially distributed with parameters $n$ and $v(x):=\int_{X:\|x-X\|\le rh} m(X)\mathrm{d}X$, 
with parameters $n$ and $v(x'):=\int_{X:\|x'-X\|\le rh} m(X)\mathrm{d}X$, respectively.
Since $B(n,v(x))B(n,v(x'))>1$ only if either $B(n,v(x))$ or $B(n,v(x'))$ is below $1$, we have
\begin{align*}
    \mathbb{E}&\left[\min\left\{1, \frac{1}{B(n,v(x))B(n,v(x'))}\right\}\right]\nonumber\\
    &\le 
    \left\{
    \mathbb{E}\left[
    \mathbbm{1}_{B(n,v(x))\le 1}+ \mathbbm{1}_{B(n,v(x')) \le 1}\right]
    +
    \mathbb{E}\left[
    \frac{\mathbbm{1}_{B(n,v(x))>1}\mathbbm{1}_{B(n,v(x'))>1}}{B(n,v(x))B(n,v(x'))}
    \right]
    \right\}.
\end{align*}
Taking an absolute constant such that $\max\{v(x),v(x')\}<C_{v}$ with $0<C_{v}<1$ (, which is possible since $h\le 1$),
we get
\begin{align}
\mathbb{E}\left[
    \mathbbm{1}_{B(n,v(x))\le 1} \right]
    &=
    (1-v(x))^{n}+n\frac{v(x)}{1-v(x)}(1-v(x))^{n}
    \nonumber\\
    &\le \left\{\frac{1}{1-C_{v}}nv(x)+1\right\}(1-v(x))^{n}
    \nonumber\\
    &\le \left\{\frac{1}{1-C_{v}}nv(x)+1\right\}\exp\{-nv(x)\}
    .
    \label{eq: bounding B1}
\end{align}
The Cauchy-Schwarz inequality gives
\begin{align}
    \mathbb{E}&\left[
    \frac{\mathbbm{1}_{B(n,v(x))>1}\mathbbm{1}_{B(n,v(x')>1}}{B(n,v(x))B(n,v(x'))}
    \right]\nonumber\\
    &\le 
    \sqrt{\mathbb{E}\left[
    \frac{\mathbbm{1}_{B(n,v(x))>1}}{B^{2}(n,v(x))}
    \right]}
    \sqrt{\mathbb{E}\left[
    \frac{\mathbbm{1}_{B(n,v(x'))>1}}{B^{2}(n,v(x'))}
    \right]}\nonumber\\
    &\le 
    \sqrt{\mathbb{E}\left[
    \frac{6 \cdot \mathbbm{1}_{B(n,v(x))>1}}{(1+B(n,v(x)))(2+B(n,v(x)))}
    \right]
    \mathbb{E}\left[
    \frac{6 \cdot \mathbbm{1}_{B(n,v(x'))>1}}{(1+B(n,v(x')))(2+B(n,v(x')))}
    \right]}\nonumber\\
    &=\frac{6}{(n+2)(n+1)v(x)v(x')},
    \label{eq: bounding 1/B}
\end{align}
where the second inequality follows since $(1+b)(2+b)\le 6b^{2}$ for $b\ge1$
and the last identity follows from a binomial calculus:
\begin{align*}
\mathbb{E}\left[
    \frac{\mathbbm{1}_{B(n,v(x))>1}}{(1+B(n,v(x)))(2+B(n,v(x)))}
    \right]
    &=\sum_{k=2}^{n}\{v(x)\}^{k}(1-v(x))^{n-k}\frac{1}{(1+k)(2+k)}\frac{n!}{k!(n-k)!}\\
    &=\frac{1}{v^{2}(x)}\frac{1}{(n+1)(n+2)}.
\end{align*}
Inequalities (\ref{eq: bounding B1}) and (\ref{eq: bounding 1/B}), together with Lemma \ref{lemma: small ball}, yield
\begin{align*}
    \mathbb{E}\left[\min\left\{1, \frac{1}{B(n,v(x))B(n,v(x'))}\right\}\right]
    \le 
    \left\{2\left(1+\frac{C'_{B}}{1-C_{v}} nh^{d}\right)
    \exp\{-C_B nh^{d}\}
    +
    \frac{6}{C_B^2}\frac{1}{n^{2}h^{2d}}
    \right\}.
\end{align*}
Taking an absolute constant $C_{7}$ such that 
\[
2\left(1+\frac{C'_{B}}{1-C_{v}} nh^{d}\right)
    \exp\{-C_B nh^{d}\}
    \le 
    C_{7}
    \frac{6}{C_B^2}\frac{1}{n^{2}h^{2d}},
\]
we obtain the desired bound under Condition \ref{condition: design} (b):
\begin{align*}
    \rho_2(n,h)
    &=
    \sup_{f \in \mathcal{F}(\beta,L)}
    \mathbb{E}(\Var(\hat{f}_{n,h}(x,x') \mid X)) 
    \le 
    C_8 n^{-2}h^{-2d} \text{ with }C_{8}:=\tau (1+C_{7}) 
    \max\left\{\frac{1}{\underk^{2}}, 1 \right\}
    \frac{6}{C^2_B}.
\end{align*}

\vspace{5mm}

\subsubsection*{Step 3: Bounding $\rho_3(n,h)$} 
We make Condition \ref{condition: design} (b)
since $\rho_{3}(n,h)=0$ under Condition \ref{condition: design} (a).

With the variance inequality $\Var[A+B]\le 2 (\Var[A]+\Var[B])$, 
we get
\begin{align*}
    \Var[\mathbb{E}[\hat{f}_{n,h}(x,x') \mid X]]
    =
    \Var\left[
    \frac{S_{n,h}}{T_{n,h}}
    \right]
    &=
    \Var\left[
        S_{n,h}\left\{ \frac{1}{\overline{T}_{n,h}} + \varepsilon_n\right\}
    \right]
    \\
    &\le 
    2\left\{
    \frac{\Var[S_{n,h}]}{\overline{T}_{n,h}^2}
    +
    \Var[S_{n,h} \varepsilon_{n,h}]
    \right\}
    \\
    &\le (2/l^2)\Var[S_{n,h}] + 2\Var[S_{n,h}\varepsilon_{n,h}],
\end{align*}
where the last inequality follows since $\overline{T}_{n,h} \ge l^2$ (by the assumption on $m$).
Admitting the evaluations
$\Var[S_{n,h}] = O(h^{2\beta}(nh^{d})^{-1} )$ and $\Var[S_{n,h} \varepsilon_n] \le \mathbb{E}(S_{n,h}^2 \varepsilon_n^2) \le \mathbb{E}(S_{n,h}^4)^{1/2}\mathbb{E}(\varepsilon_n^4)^{1/2}=O(h^{2\beta}(nh^{d})^{-1} )$ proved by Lemmas \ref{lemma: moment of epsilon}, \ref{lemma:var_snh}, and \ref{lemma:var_snhenh}, 
we obtain
\[
    \Var[\mathbb{E}[\hat{f}_{n,h}(x,x') \mid X]]
    \le
    C_9 h^{2\beta}(nh^{d})^{-1}
    \text{ for some $C_9>0$ not depending on $n$ and $h$},
\]
which completes the proof. 
\qed

\vspace{3mm}

\subsection{Proof of Theorem \ref{theorem: lower bound}}
\label{app:proof of lower}

Without loss of generality, we can assume $x=x'=0$.

Let us fix $\phi\in \mathcal{F}(\beta,1)$ satisfying the following condition:
\begin{itemize}
    \item $\phi(0,0)=1$;
    \item $\phi(x,x')>0$ if and only if $\|x\| \le 1$ and $\|x'\|\le 1$;
    \item $\sup_{x,x'}\phi(x,x) \le 1$.
\end{itemize}
Take a one-parameter subset $\tilde{\mathcal{F}}$ of $\mathcal{F}(\beta,L)$ in such a way that
\begin{align*}
    \tilde{\mathcal{F}} := \left\{ f(\cdot,\cdot)=\eta \phi(\cdot/h,\cdot/h)\,:\, |\eta|\le Lh^{\beta} \right\}.
\end{align*}
For any fixed $f(\cdot,\cdot)=\eta \phi(\cdot/h,\cdot/h) \in \tilde{\mathcal{F}}$, we define its estimator  $\hat{f}(\cdot,\cdot)=\hat{\eta}\phi(\cdot/h,\cdot/h)$, where the real value $\hat{\eta}$ is specified by the observations $X_1,\ldots,X_n$ and $Y_{12},Y_{13},\ldots,Y_{(n-1)n}$. 
As $\hat{\eta}$ is a function of the observations, a set of functions $\hat{\eta}$ is denoted by $\mathcal{R}$. 
Observe that
\begin{align}
\inf_{\hat{f}}\sup_{f\in \mathcal{F}(\beta,L)}\mathbb{E}|f(0,0)-\hat{f}(0,0)|^{2}
&\ge 
\inf_{\hat{f}}\sup_{f \in \tilde{\mathcal{F}}}
\mathbb{E}|f(0,0)-\hat{f}(0,0)|^{2}\nonumber\\
&=\inf_{\hat{\eta} \in \mathcal{R}}\sup_{\eta: |\eta|\le Lh^{\beta}}
\mathbb{E}|\eta-\hat{\eta}|^{2}\nonumber\\
&\ge \inf_{\hat{\eta} \in \mathcal{R}} \int \mathbb{E}|\eta-\hat{\eta}|^{2} \pi(\eta)d\eta,
\label{eq: Bayes risk lower bound}
\end{align}
where the last inequality follows since the average is bounded above by the maximum. Here consider bounding the right-most side in (\ref{eq: Bayes risk lower bound}).
To do so, we employ the van Tree inequality.
    \begin{lemma}[The van Tree inequality; \citealp{vanTreesbook,GillandLevit(1995)}]
    \label{lem: van tree}
    Let $\{p(z\mid \theta):\theta\in\Theta\}$ be a parametric model with $\Theta$ a closed interval on the real line. Let $\pi(\theta)$ be a probability density on $\Theta$ that converges to zero at the endpoints of the interval $\Theta$. Let $\hat{\theta}$ be any estimator of $\theta$.
    If $p(z\mid\theta)$ satisfies
    \begin{align*}
        \Ep_{\theta}(\partial/\partial\theta)\{\log p(Z\mid\theta)\}=0,
    \end{align*}
    then we have
    \begin{align*}
        \Ep_{\theta}\{\hat{\theta}(Z)-\theta\}^{2}
        \ge \frac{1}{\int[\mathcal{I}(\theta)]\pi(\theta)d\theta+\mathcal{I}(\pi)},
    \end{align*}
    where let
    \begin{align*}
        \mathcal{I}(\theta)&:=\Ep_{\theta}[\{(\partial/\partial\theta)\log p(Z\mid\theta)\}^{2}]\quad\text{and}\\
        \mathcal{I}(\pi)&:=\int \{(\partial/\partial \theta) \log \pi(\theta)\}^{2}\pi(\theta)d\theta.
    \end{align*}
    \end{lemma}
    The derivation is given in \citet{GillandLevit(1995)}.
Let $q(Y\mid\mu)$ denote the conditional density of $Y$ given the conditional mean $\mu$.
By setting $z=(y_{12},\ldots,y_{(n-1)n})$,
    $\theta=\eta$,
    $\Theta=\{\eta:|\eta|\le Lh^{\beta}\}$,
    and
    $p(z\mid\theta)=\prod_{i_{i}<i_{2}}q(y_{i_{1}i_{2}}\mid \eta \phi(X_{i_{1}},X_{i_{2}}))$,
and by taking arbitrary prior density satisfying the assumption in Lemma \ref{lem: van tree} as $\pi$,
    the van Tree inequality gives
\begin{align}
\inf_{\hat{\eta} \in \mathcal{R}} \int \mathbb{E}[|\eta-\hat{\eta}|^{2} \mid X_{1},\ldots,X_{n}]\pi(\eta)d\eta
\ge \frac{1}{\int \mathcal{I}_{n}(\eta) \pi(\eta)d\eta + \mathcal{I}_{\pi}},
\label{eq: van Trees inequality}
\end{align}
where 
\begin{align*}
\mathcal{I}_{n}(\eta):= \mathbb{E}
\left\{ \sum_{i_{1}<i_{2}} \partial_{\eta} \log q(Y_{i_{1}i_{2}}\mid \eta\phi(X_{i_{1}},X_{i_{2}})) \right\}^{2}
\text{ and }
\mathcal{I}_{\pi}:=\int \{\partial_{\eta}\log \pi(\eta)\}^{2}\pi(\eta)d\eta.
\end{align*}
Combining (\ref{eq: Bayes risk lower bound}) and (\ref{eq: van Trees inequality}) with the Jensen inequality yields
\begin{align}
    \inf_{\hat{f}}\sup_{f\in \mathcal{F}(\beta,L)}\mathbb{E}|f(0,0)-\hat{f}(0,0)|^{2}
&\ge \frac{1}{\Ep_{X_{1},\ldots,X_{n}}[\int \mathcal{I}_{n}(\eta)]+\mathcal{I}_{\pi}}. 
\label{eq: last bound}
\end{align}
By a change of variables ($\eta\mapsto \eta/h^{\beta}$), 
we get
\begin{align}
    \mathcal{I}_{\pi} &= \left(\frac{L}{h^{\beta}}\right)^{2}\mathcal{I}_{\tilde{\pi}},
\label{eq: bound Ipi}
\end{align}
where $\tilde{\pi}$ is a density of $\eta/h^{\beta}$.
By a change of variables ($\eta \mapsto \mu:= \eta\phi(X_{i_{1}}/h, X_{i_{2}}/h) $), we get 
\begin{align*}
    \mathcal{I}_{n}(\eta) &= \sum_{i_{1}<i_{2}} 
    \{\phi(X_{i_{1}}/h,X_{i_{2}}/h)\}^{2}
    \mathbb{E}\{\partial_{\mu}\log q(Y\mid \mu)\}^{2} \nonumber\\
    &= \mathcal{I}(\mu) \sum_{i_{1}<i_{2}}
    \{\phi(X_{i_{1}}/h,X_{i_{2}}/h)\}^{2}
\end{align*}
where $\mathcal{I}(\mu)=\mathbb{E}\{\partial_{\mu}\log q(Y\mid \mu)\}^{2}$.
In fixed-design cases, 
from the condition that $\phi(x,x')\le 1$ and
from Condition \ref{condition: design} (a), we get
\begin{align}
\sum_{i_{1}<i_{2}}
    \{\phi(X_{i_{1}}/h,X_{i_{2}}/h)\}^{2}
    &\le \left|\left\{t\in\mathbb{Z}^{d}\mid \left\|\frac{c_{X}/n^{1/d}t}{h}\right\|\le  1\right\}\right|^{2}\nonumber\\
    &\le \frac{1}{c_{X}^{d}} n^{2}h^{2d}.
\label{eq: bound In fixed-design}
\end{align}

In random-design cases, 
from the condition that $\phi(x,x')\le 1$ and
from Lemma \ref{lemma: common bound}, we get
\begin{align}
\Ep_{X_{1},\ldots,X_{n}}\sum_{i_{1}<i_{2}}
    \{\phi(X_{i_{1}}/h,X_{i_{2}}/h)\}^{2}
    \le n^{2} \Pr(\|X_{1}/h\|\le 1)
    \Pr(\|X_{1}/h\|\le 1)
    \le (C'_B)^{2}n^{2}h^{2d}.
\label{eq: bound In random-design}
\end{align}
Thus, 
substituting (\ref{eq: bound In fixed-design}) or (\ref{eq: bound In random-design}), and (\ref{eq: bound Ipi}) into (\ref{eq: last bound}) yields
\begin{align*}
\inf_{\hat{f}}\sup_{f\in \mathcal{F}(\beta,L)}\mathbb{E}|f(0,0)-\hat{f}(0,0)|^{2}
\ge \frac{1}{n^{2}h^{2d}C_{l} \int \mathcal{I}(\mu)\pi(\mu)d\mu + h^{-2\beta}\mathcal{I}_{\tilde{\pi}} },
\end{align*}
where $C_{l}$ is a positive constant independent from $n$ and $h$.
Together with $h=n^{-1/(\beta+d)}$, this gives the desired inequality and completes the proof. \qed


\end{appendices}

\end{document}